\documentclass[12pt,reqno]{amsart}
\usepackage[headings]{fullpage}
\usepackage{amssymb,amsmath,amscd,tikz,mathtools}
\usepackage{graphicx}
\usepackage{texdraw}
\usepackage[all,cmtip]{xy}
\usepackage{url}
\usepackage[bookmarks=true,%
    colorlinks=true,%
    linkcolor=blue,%
    citecolor=blue,%
    filecolor=blue,%
    menucolor=blue,%
    urlcolor=blue,%
    breaklinks=true]{hyperref}

\usepackage{slashed}
\usepackage{tikz-cd}
\usepackage{stmaryrd}
\usepackage{adjustbox}

\usepackage{verbatim}
\usetikzlibrary{hobby}
\usetikzlibrary{decorations.pathmorphing}
\tikzset{snake it/.style={decorate, decoration=snake}}
\usetikzlibrary{patterns}

\newtheorem{theorem}{Theorem}[section]
\theoremstyle{definition}
\newtheorem{proposition}[theorem]{Proposition}
\newtheorem{lemma}[theorem]{Lemma}
\newtheorem{definition}[theorem]{Definition}
\newtheorem{remark}[theorem]{Remark}
\newtheorem{corollary}[theorem]{Corollary}
\newtheorem{conjecture}[theorem]{Conjecture}

\newtheorem{example}[theorem]{Example}

\def\BZ{\mathbb Z}
\def\BQ{\mathbb Q}

\def\BC{\mathbb C}
\def\BK{\mathbb K}

\def\calA{\mathcal A}

\def\calS{\mathcal S}

\def\calL{\mathcal L}

\def\calH{\mathcal H}

\def\calR{\mathcal R}
\def\calO{\mathcal O}
\def\s{\sigma}

\def\Res{\mathrm{Res}}

\def\a{\alpha}

\def\d{\delta}
\def\ve{\varepsilon}
\def\th{\theta}

\def\Li{\mathrm{Li}}

\def\be{\begin{equation}}
\def\ee{\end{equation}}

\def\z{\zeta}

\def\diag{\mathrm{diag}}

\usepackage{accents}

\def\={\;=\;}
\def\SU{\mathrm{SU}}
\def\ZHS{\BZ\mathrm{HS}}
\def\QHS{\BQ\mathrm{HS}}

\def\vphi{\varphi}
\def\loop{\mathrm{loop}}

\def\fsl{\mathfrak{sl}}
\def\End{\mathrm{End}}

\def\fb{\mathfrak{b}}
\def\tr{\mathrm{tr}}
\def\trim{\mathrm{trim}}

\begin{document}


\title[A lift of the colored Jones polynomial of a knot]{
  A lift of the colored Jones polynomial of a knot}

\author{Stavros Garoufalidis}
\address{
  International Center for Mathematics, Department of Mathematics \\
  Southern University of Science and Technology \\
  Shenzhen, China \newline
  {\tt \url{http://people.mpim-bonn.mpg.de/stavros}}}
\email{stavros@mpim-bonn.mpg.de}

\author{Campbell Wheeler}
\address{Institut des Hautes Études Scientifiques\\ \newline
         35 rte de Chartres, 91440 Bures-sur-Yvette, France \newline
         {\tt \url{https://www.ihes.fr/~wheeler/}}}
\email{wheeler@ihes.fr}


\date{27 February 2026}

\begin{abstract}
  Habiro lifted the Witten--Reshetikhin--Turaev invariant
  of an integer homology 3-sphere (a complex-valued function on the set
  of complex roots of unity) to an element of the Habiro ring.
  We lift the colored Jones polynomial of a knot, with Alexander polynomial
  $\Delta(t)$, to the recently introduced Habiro ring of the \'etale map
  $\BZ[t^{\pm 1}]\to \BZ[t^{\pm 1},\Delta(t)^{-1}]$ (with Frobenius lifts $t\mapsto t^p$ for all primes $p$). This implies the existence
  of a loop expansion at roots of unity (confirming a conjecture of Habiro),
  and a lift of power series invariants of Ohtsuki for 3-manifolds with Betti number 1
  to a Habiro ring. Our results have natural extensions to the skein module of
  a knot complement, and they  suggest a natural lift of the colored Jones polynomial
  colored by representations of a simple Lie (super) algebra. 
\end{abstract}

\maketitle

{\footnotesize
\tableofcontents
}


\section{Introduction}
\label{sec.intro}

\subsection{TQFT and the Habiro ring}
\label{sub.comp}

The discovery of the Jones polynomial of a knot or a link in $S^3$ by
V. Jones~\cite{Jones} and its interpretation as an observable of a topological
quantum field theory (a gauge theory with action the Chern-Simons function)
by Witten~\cite{Witten:CS} led in the notion of a topological quantum field theory
(in short, TQFT) by Atiyah~\cite{Atiyah}.

TQFT is a fruitful way to organize complex information about manifolds in dimensions
0,1,2 and 3, and known constructions use for instance a compact Lie group (such
as $\SU(2)$) and a complex root of unity $\z$. TQFTs with such data were constructed
originally by Reshetikhin--Turaev~\cite{RT:ribbon}, and subsequently by many authors.

These theories assign numerical invariants to closed 3-manifolds, finite-dimensional
vector spaces to closed surfaces, categories to closed 1-manifolds that satisfy rather
elaborate cut-and-paste properties that also allow manifolds with boundaries and
corners.

In particular, the WRT invariants of a closed 3-manifold is a complex-valued function
defined on the set $\mu$ of complex roots of unity.
At the time, it was remarked by Reshetikhin--Turaev that they did not know if their
invariants were the evaluation of a polynomial at roots of unity, like for the
Jones polynomial. A key discovery of Habiro~\cite{Habiro:WRT} was that this was not
the case and instead he found a lift of the WRT invariant $Z_M(\z)$
of an integer homology 3-sphere $M$ to an element $f_M(q)$ of the Habiro ring
\be
\label{habdef}
\calH \= \varprojlim_n \;\BZ[q]/(q;q)_n\BZ[q]
\ee
(where $(q;q)_n=(1-q) \dots (1-q^n)$ is the $q$-Pochhammer symbol), such that
for every $\z \in \mu$, we have $Z_M(\z)=f_M(\z)$.

Elements of the Habiro ring can be evaluated not only at roots of unity, but also
near roots of unity leading to homomorphisms 
\be
\label{f2fm}
f(q) \in \calH \to f_m(q-\z_m) \in \BZ[\z_m][\![q-\z_m]\!]
\ee
where $(\z_m)_{m \geq 1}$ is a compatible collection of complex roots of unity 
of order $m$, that is a collection satisfying
\be
\label{zmdef}
\z_{mm'} \=\z_m\z_{m'}, \qquad (m,m')\=1, \qquad
(\z_{p^r})^p \=\z_{p^{r-1}}
\ee
for all positive integers $m$ and $m'$, primes $p$ and positive integers $r$.

The family of homomorphisms given in~\eqref{f2fm} satisfy a $p$-adic gluing property,
namely 
\be
\label{gluefmp}
f_m(q-\z_{pm}+\z_{pm}-\z_m)\=f_{pm}(q-\z_{pm}) \in\BZ_p[\z_{pm}][\![q-\z_{pm}]\!],
\qquad (m \in \BZ_{>0}, \,\, p \,\, \text{prime})
\ee
where $\BZ_p=\lim_n \BZ/(p^n)$ is the ring of $p$-adic integers.
Equation~\eqref{gluefmp} makes sense because of the crucial fact that, for a suitable choice of $\{\z_m\}_m$, we have
$|\z_{pm}-\z_m|_p<1$ where $|\cdot|_p$ denotes the $p$-qdic norm on $\BQ_p$. 

Putting everything together, Habiro proves in~\cite{Habiro:completion}
that there is an isomorphism $\calH \cong \calH_\BZ$ where 
\be
\label{habdefZ}
\begin{aligned}
  \calH_\BZ
  \= \Big\{&(f_m(q-\z_m))_m\;\in\;\prod_{m\geq 1}
  \BZ[\z_m][\![q-\z_m]\!]\;\Big|\; \eqref{gluefmp} \,\, \text{holds} \Big\} \,.
\end{aligned}
\ee

\subsection{Rings for TQFT invariants}
\label{sub.rings}

Recently, motivated by arithmetic properties of perturbative Chern--Simons theory,
a Habiro ring for an \'etale $\BZ[t^{\pm1}]$-algebra $R$ was implicitly discussed in
~\cite{GSWZ}, introduced in Scholze's course~\cite{Scholze}
and further studied in~\cite{GW:explicit}. The definition of $\calH_{R/\BZ[t^{\pm1}]}$
is similar to~\eqref{habdefZ} where we replace $\BZ$ by $R$, $\BZ_p$ by
$R_p=\lim_n R/(p^n)$, and crucially, we insert in the right hand side of
Equation~\eqref{gluefmp} the $p$-Frobenius $\vphi_p$ with $\vphi_p(t)=t^p$.

\begin{definition}
\label{def.habR}  
The Habiro ring of $R/\BZ[t^{\pm1}]$, is defined to be
\be
\begin{aligned}
  \calH_{R}
  \=\Big\{&(f_m)_m\;\in\;\prod_{m\geq 1}R[\z_m,t^{1/m}][\![q-\z_m]\!]
  \;\Big|\;\text{for all }m\in\BZ_{>0}\text{ and }p \,\text{ prime}\\
  &f_m(x,q-\z_{pm}+\z_{pm}-\z_m)\=\varphi_pf_{pm}(x^p,q-\z_{pm})\Big\}\,,
\end{aligned}
\ee
where $\varphi_p(q)=q$ and $\varphi_p(\z)=\z$.
\end{definition}
The action $\vphi$ of the Frobenius makes $\calH_{R}$ an interesting and nontrivial
ring.
This is also where perturbative Chern--Simons invariants naturally live~\cite{GSWZ},
as opposed to the completion $\varprojlim_n \;R[q]/(q;q)_nR[q]$.

These rings are also fundamentally different from the cyclotomic expansion
\be
\varprojlim_n \BZ[q^{\pm 1},t^{\pm1}]/((t^{-1} q^a;q)_n  \, | \, a \in \BZ) \,.
\ee
Indeed, they contain strictly more information as can be seen from the following:

\begin{lemma}
\label{lem.unique}
If the ring homomorphism $\mathrm{ev}_{t=1}:\BZ[t^{\pm1}]\to\BZ$ extends to a map $\mathrm{ev}_{t=1}:R\to\BZ$, then there is an injective map
\be
\label{eq.unique}
  \calH_{R/\BZ[t^{\pm 1}]}
  \hookrightarrow
  \varprojlim_n
  \calH_{\BZ}[t^{\pm1}]/(t^{-1};q)_n\,.
\ee
Moreover, if $f(t,q) \in \calH_{R/\BZ[t^{\pm 1}]}$
satisfies $f(q^{n-1},q) \in \BZ[q^{\pm 1}]$ for all $n>0$,
then its image is in the subring given by $\varprojlim_n
\BZ[t^{\pm1},q^{\pm1}]/(t^{-1};q)_n$.
\end{lemma}

\subsection{Our results}
\label{sub.results}

Recall that the colored Jones polynomial $J_{K,n}(q) \in \BZ[q^{\pm 1}]$ of a knot $K$
is a sequence of Laurent polynomials for $n \geq 1$ that, roughly speaking, encodes
the knot observed by the $(n+1)$-dimensional representation of $\SU(2)$,
or alternatively, by the Jones polynomial of a knot and its parallels~\cite{RT:ribbon}.

Our theorem lifts the colored Jones polynomial of a knot to an element
of the Habiro ring $\calH_{\BZ[t^{\pm 1}, \Delta_K(t)^{-1}]/\BZ[t^{\pm 1}]}$, where
$\Delta_K(t) \in \BZ[t^{\pm 1}]$ is the Alexander polynomial of $K$, which measures
the homology of the universal abelian cover of the knot complement~\cite{Alexander}.
By functoriality and the fact that $\Delta_K(1)=1$, one has a map 
\be
\label{maptn}
\calH_{\BZ[t^{\pm 1}, \Delta_K(t)^{-1}]/\BZ[t^{\pm 1}]} \to \calH_\BZ\,,
\qquad t \mapsto q^{n} \qquad (n \in \BZ) \,.
\ee

\begin{theorem}
\label{thm.1}
There is a map
\be
\label{JKtq}
J: \{ \text{Knots }K\text{ in} \,\, S^3\}
  \to \calH_{\BZ[t^{\pm 1}, \Delta_K(t)^{-1}]/\BZ[t^{\pm 1}]}\,,
\qquad K \mapsto J_K(t,q) 
\ee
satisfying
\be
\label{eq1}
J_K(q^{n-1},q) \= J_{K,n}(q) \in \BZ[q^{\pm 1}] \subset
\calH_\BZ \hookrightarrow \BZ[\![q-1]\!]
\qquad (\text{for all $n>0$}) \,.
\ee
\end{theorem}

The above theorem does not use Habiro's theorems from~\cite{Habiro:WRT} about the
colored Jones polynomial, and has some immediate corollaries.

The first concerns the cyclotomic expansion of the colored Jones polynomial,
originally proven by Habiro~\cite{Habiro:WRT}. 

\begin{corollary}
\label{cor.cyclo}
For every knot $K$ in $S^3$, there exists a sequence of polynomials
$a_{K,k}(q) \in \BZ[q^{\pm 1}]$ such that for all $n>0$, we have
\be
\label{eq.cycloj}
J_{K,n}(q) \= \sum_{k=0}^\infty a_{K,k}(q) (q^{-n+1};q)_k \,.
\ee
\end{corollary}

The next corollary concerns the so-called loop expansion of the colored Jones 
polynomial introduced by Rozansky~\cite{Rozansky}
\be
\label{Jloop}
J^\loop_K(t,q) \in \BZ[t^{\pm 1}, \Delta(t)^{-1}][\![q-1]\!], \qquad
J^\loop_K(t,1) \= \frac{1}{\Delta_K(t)}
\ee
uniquely characterized by the fact that
\be
\label{uniqueJ}
J_{K,n}(q) \= J^\loop_K(q^{n-1},q) \in \BQ[n][\![q-1]\!], \qquad (n > 0) \,.
\ee
In~\cite[Conj.7.4]{Habiro:WRT}, Habiro conjectured that the loop expansion
can be defined near all roots of unity $\z_m$. 
This is built into the very definition of the Habiro ring of $\BZ[t^{\pm 1}, \Delta(t)^{-1}]$.
Recall that an element
$J_K(t,q) \in \calH_{\BZ[t^{\pm 1}, \Delta(t)^{-1}]/\BZ[t^{\pm 1}]}$
is a collection of series
$J_{K,m}(t,q-\z_m) \in \BZ[t^{\pm 1}, \Delta(t)^{-1}][t^{1/m},\z_m][\![q-\z_m]\!]$ 
that glue, after a Frobenius twist.

\begin{corollary}
\label{cor.hab}
For every knot $K$ in $S^3$, we have
\be
\label{eq2}
J_{K,1}(t,q-1) \= J^\loop_K(t,q) \in \BZ[t^{\pm 1}, \Delta(t)^{-1}][\![q-1]\!] \,,
\ee
and the collection of series $J_{K,m}(t,q-\z_m)$ at roots of unity exists (which proves
Habiro's conjecture).
Moreover, these expansions at roots of unity determine
one another after gluing with Frobenius twists. 
\end{corollary}


The next corollary concerns a compatibility between geometric Dehn-filling operations
on a knot complement, and algebraic push-forward operations on the Habiro ring of
Theorem~\ref{thm.1}.
Recall that Dehn filling with slope $\gamma \in \BQ \cup \{\infty\}$ is the process
of gluing a solid torus $S^2 \times D^2$ to the boundary of a 3-manifold with torus
boundary using a diffeomorphism of the boundary torus in such a way that the curve
with slope $\gamma$ bounds a disk. Applying this operation to a knot complement
$S^3\setminus K$ results to a closed 3-manifold $S^3_{K,\gamma}$ whose homology
is given by $H_1(S^3_{K,a/b},\BZ)=\BZ/a\BZ$ for coprime integers $(a,b)$. 

As mentioned in the introduction, Habiro lifted the WRT invariant of an integer
homology sphere to a map
\be
\label{ZHS}
\{\ZHS\} \to \calH_\BZ \,.
\ee
This map was extended to a lift of the WRT invariant of rational homology 3-spheres
$M$ with $|H_1(M,\BZ)|=N$ by Beliakova--B\"uhler--L\^{e}~\cite{BL:SO3}.

On the side of algebra, we use Lemma~\ref{lem.unique} and~\cite[Sec. 3]{BCL}
to define the twisted push-forward maps 
\be
\label{Omdef}
\begin{aligned}
\calL_1 : &  \calH_{\BZ[t^{\pm 1}, \Delta(t)^{-1}]/\BZ[t^{\pm 1}]} \to \calH_\BZ,
\qquad t^k \mapsto q^{-k^2}
\\
\calL_{-1} : & \calH_{\BZ[t^{\pm 1}, \Delta(t)^{-1}]/\BZ[t^{\pm 1}]} \to \calH_\BZ,
\qquad t^k \mapsto q^{k^2} \,.
\end{aligned}
\ee
Note that in~\cite{BL} a surgery operation was defined for $\gamma \in \BQ$ with
numerator $1$, but for simplicity, we will not discuss it here.
We will show below that these maps are well-defined.
The main theorem of Beliakova--L\^{e} together with the isomorphism
$\calH_{\BZ[N^{-1}]/\BZ} \cong \varprojlim_n \BZ[N^{-1}][q]/(q;q)_n$
imply the following: 

\begin{corollary}
\label{cor.lap1}
For every $\gamma \neq 0$, we have a commutative diagram
\be
\label{cd2}
\begin{tikzpicture}
\draw(0,2) node {$\{\mathrm{Knots}\}$};
\draw(4,2) node {$\calH_{\BZ[t^{\pm 1}, \Delta(t)^{-1}]/\BZ[t^{\pm 1}]}$};
\draw(0,0) node {$\{\QHS\}$};
\draw(4,0) node {$\calH_{\BZ[N^{-1}]/\BZ}$};
\draw[->](1,2)--(2.2,2);
\draw(1.6,2.3) node {$J$};
\draw[->](0,1.5)--(0,0.5);
\draw(-0.6,1) node {$\mathrm{Dehn}_\gamma$};
\draw[->](4,1.5)--(4,0.5);
\draw(4.3,1) node {$\calL_\gamma$};
\draw[->](0.8,0)--(3,0);
\end{tikzpicture}
\ee
\end{corollary}  

The next corollary concerns WRT invariants of 3-manifolds $M$ with $H_1(M,\BZ)=\BZ$
(i.e., closed 3-manifolds with Betti number 1 and torsionfree homology). Such
manifolds $M$ have a well-defined Alexander polynomial $\Delta_M(t) \in \BZ[t^{\pm 1}]$,
and typical examples include the manifolds obtained by 0-surgery on a knot in $S^3$.

In~\cite{Ohtsuki}, Ohtsuki extracted power series invariants from the WRT invariants
\be
\label{b1}
\{M \,\, | \,\, H_1(M,\BZ)=\BZ \} \to \BZ[\delta_M^{-1}][\![q-1]\!], \qquad
\ee
where $\delta_M$ is the leading coefficient of $\Delta_M(t)$, as well as
invariants 
\be
\label{b11}
\{M \,\, | \,\, H_1(M,\BZ)=\BZ \} \to \calO_\BK[\mathrm{disc}(\BK)^{-1}][\![q-1]\!]
\ee
where $\a$ is a root of the Alexander polynomial and $\BK=\BQ(\a)$ is the
corresponding number field of discriminant $\mathrm{disc}(\BK)$ and ring of integers $\calO_\BK$. 

Our next corollary lifts the above-mentioned invariants of Ohtsuki for the manifolds
obtained by 0-surgery on a knot in $S^3$, to a Habiro ring.

\begin{corollary}
\label{cor.0surgery}  
We have a commutative diagram 
\be
\label{cd3}
\begin{tikzpicture}
\draw(0,2) node {$\{\mathrm{Knots}\}$};
\draw(4,2) node {$\calH_{\BZ[t^{\pm 1}, \Delta(t)^{-1}]/\BZ[t^{\pm 1}]}$};
\draw(0,0) node {$\{M \,\, | \,\, H_1(M,\BZ)=\BZ \}$};
\draw(4,0) node {$\calH_{\calO_\BK[\mathrm{disc}(\BK)^{-1}]/\BZ}$};
\draw[->](1,2)--(2.2,2);
\draw(1.6,2.3) node {$J$};
\draw[->](0,1.5)--(0,0.5);
\draw(-0.6,1) node {$\mathrm{Dehn}_0$};
\draw[->](4,1.5)--(4,0.5);
\draw(4.3,1) node {$\calL_0$};
\draw[->](2.1,0)--(2.6,0);
\end{tikzpicture}
\ee
\end{corollary}
In particular, one obtains topological invariants in $\calO_\BK[\mathrm{disc}(\BK)^{-1}][\z_m]$ of
closed 3-manifolds $M$ as above for all roots of unity $\z_m$ of order coprime to $\mathrm{disc}(\BK)$.


\begin{remark}
\label{rem.notiso}
Note that
\be
\label{caution}
\calH_{\calO_\BK[\mathrm{disc}(\BK)^{-1}]/\BZ} \neq \varprojlim_n \calO_\BK[\mathrm{disc}(\BK)^{-1}][q]/(q;q)_n,
\quad
\calH_{\calO_\BK[\mathrm{disc}(\BK)^{-1}]/\BZ} \neq \calH_\BZ \otimes_{\BZ}\calO_\BK[\mathrm{disc}(\BK)^{-1}] 
\ee
and the fact that the Ohtsuki invariants give well-defined elements of
$\calH_{O_\BK[\mathrm{disc}(\BK)^{-1}]/\BZ}$ is a non-trivial statement. Other elements of such
rings were constructed in the main results of ~\cite{GSWZ} using Chern--Simons
perturbation theory. 
\end{remark}

Our next corollary involves $q$-holonomic aspects of the lift of the colored
Jones polynomial. Observe that the map $\s$ that shifts $t$ to $qt$ is an
ring automorphism
\be
\label{Hsigma}
\s: \calH_{\BZ[t^{\pm 1}, \Delta(t)^{-1}]/\BZ[t^{\pm 1}]} \to
\calH_{\BZ[t^{\pm 1}, \Delta(t)^{-1}]/\BZ[t^{\pm 1}]}, \qquad t \mapsto q t \,.
\ee
Using $\s$, we can define a notion of $q$-holonomicity for elements of the above
ring. The $q$-hypergeometric origin of $J_K(t,q)$, or the fact that the colored
Jones polynomial is $q$-holonomic~\cite{GL:qholo} combined with
Equation~\eqref{uniqueJ}, imply the following:

\begin{corollary}
\label{cor.Jqholo}  
For every knot $K$, $J_K(t,q)$ is $q$-holonomic and in fact satisfies the same
$q$-difference equations as $(J_{K,n}(q))_{n \geq 1}$. 
\end{corollary}

We end this subsection with a parenthetical remark.

\begin{remark}
\label{rem.maptn2}
The map~\eqref{maptn} is actually a composition of $\s^n$ followed by the
evaluation at $t=1$, where
\be
\label{t=1}
\calH_{\BZ[t^{\pm 1}, \Delta(t)^{-1}]/\BZ[t^{\pm 1}]} \to \calH_\BZ, \qquad
f(t,q) \mapsto f(1,q) \,.
\ee
\end{remark}

\subsection{Extension to skein modules}
\label{sub.ext}

Theorem~\ref{thm.1} has an extension to the Kauffman bracket skein module
$\calS_q(S^3\setminus K)$ of the 3-manifold $M=S^3\setminus K$ with torus boundary. 
The latter, going back to Przytycki~\cite{Przytycki} and Turaev~\cite{Tu:conway}
in the 1990s, is a (non-finitely generated) $\BZ[q^{\pm 1/4}]$-module associated
to a 3-manifold whose generators are framed links, in the spirit of Conway's
interpretation of knot invariants and the skein theory of the Jones
polynomial. The $4$th root of $q$ is a technical issue (forcing one to invert the prime 2) and can be avoided by
considering the even skein module $\calS^{\mathrm{ev}}_q(M) \subset \calS(M)$ of $M$
generated by links with even color (i.e., color from the root lattice of
$\fsl_2(\BC)$), or equivalently from the 2-parallels of all framed links in $M$;
see~\cite{GL:from}. 

It follows from the proof of Theorem~\ref{thm.1} that Equation~\eqref{eq1}
extends to to a map
\be
\label{eq1e}
\calS^{\mathrm{ev}}_q(S^3\setminus K)
\to \calH_{\BZ[t^{\pm 1}, \Delta_K(t)^{-1}]/\BZ[t^{\pm 1}]}
\ee
from the skein module of the knot complement $M=S^3\setminus K$ so that
$L \subset S^3\setminus K$ goes to the colored Jones polynomial of $K \cup L$ where
we color $K$ with the $n$-dimensional representation of $\fsl_2(\BC)$ and all
components of $L$ with the 2-dimensional representation of $\fsl_2(\BC)$.

Phrased this way, this extended theorem is a relative version
of~\cite[Thm.1.1]{GL:skein} for 3-manifolds with torus boundary, and fits in
the following commutative diagram for a fixed knot $K$ in $S^3$ and for
$\gamma \neq 0$.
\be
\label{cd4}
\begin{tikzpicture}
\draw(0,2) node {$\calS^{\mathrm{ev}}_q(S^3\setminus K)$};
\draw(4,2) node {$\calH_{\BZ[t^{\pm 1}, \Delta_K(t)^{-1}]/\BZ[t^{\pm 1}]}$};
\draw(0,0) node {$\calS_q(S^3_{K,\gamma})$};
\draw(4,0) node {$\calH_{\BZ[N^{-1}]/\BZ}$};
\draw[->](1.2,2)--(2.2,2);
\draw(1.6,2.3) node {$J$};
\draw[->](0,1.5)--(0,0.5);
\draw(-0.6,1) node {$\mathrm{Dehn}_\gamma$};
\draw[->](4,1.5)--(4,0.5);
\draw(4.3,1) node {$\calL_\gamma$};
\draw[->](1,0)--(3,0);
\end{tikzpicture}
\ee

\subsection{Knot polynomials associated to Lie algebras}
\label{sub.lie}

Theorem~\ref{thm.1} points to an obvious extension of the quantum knot
polynomials $J^{\mathfrak{g}}_{K,\lambda}(q) \in \BZ[q^{\pm 1}]$ defined via
representations of highest weight $\lambda$ of a simple Lie algebra
$\mathfrak{g}$, and after a rescaling of $q \mapsto q^{a_\mathfrak{g}}$
for a nonzero rational number $a_\mathfrak{g}$ that only depends
on the Lie algebra but not on the knot or the representation, see~\cite{Lusztig}.

Fix simple roots $\alpha_j$ for $j=1,\dots,r$, where $r$ is the rank of
$\mathfrak{g}$ and consider variables $t=(t_1,\dots,t_r)$ and monomials 
$t_\alpha=\prod_{i=1}^r t_i^{c_i}$ for every positive root
$\alpha=\sum_{i=1}^r c_i \alpha_i$. 
If $\lambda$ is a dominant weight, the specialization $t=q^\lambda$ means that
$t_\alpha=q^{(\lambda,\alpha)}$ for all positive roots $\alpha$.

The proof of the MMR Conjecture~\cite{BG:MMR} states that for every simple Lie
algebra $\mathfrak{g}$ and every dominant weight $\lambda$, we have
\be
\label{MMR}
J^{\mathfrak{g}}_{K,\lambda}(q) \= \frac{1}{\Delta^{\mathfrak{g}}(q^\lambda)} + O(q-1)
\ee
where $\Delta^{\mathfrak{g}}(t) = \prod_{\alpha >0 } \Delta_K(t_\alpha)$.
Combining the MMR Conjecture with Theorem~\ref{thm.1} for
$\mathfrak{g}=\mathfrak{sl}_2$, one arrives at a clean lift of
the colored Jones polynomial of a simple Lie algebra.

\begin{conjecture}
\label{conj.1g} 
There is a map
\be
\label{Jg}
\{ \text{Knots }K\text{ in} \,\, S^3\}
  \to \calH_{\BZ[t^{\pm 1}, \Delta_K^{\mathfrak{g}}(t)^{-1}]/\BZ[t^{\pm 1}]},
\qquad K \mapsto J^\mathfrak{g}_K(t,q) 
\ee
satisfying
\be
\label{eg}
J^{\mathfrak{g}}_K(q^\lambda,q)
= J^{\mathfrak{g}}_{K,\lambda}(q) \in \BZ[q^{\pm 1}]
\subset \calH_\BZ \hookrightarrow \BZ[\![q-1]\!],
\qquad (\text{for all dominant $\lambda$}) \,.
\ee
\end{conjecture}

The above conjecture should imply the analogous phrasings of Corollaries
~\ref{cor.cyclo}, \ref{cor.hab}, \ref{cor.lap1} and \ref{cor.0surgery}
for all simple Lie alegbras. In a sense, the above conjecture bypasses the question
of integral bases for the representations of the quantum group associated to a simple
Lie algebra $\mathfrak{g}$ that generalize the basis that Habiro constructs for
$\mathfrak{sl}_2$~\cite{Habiro:WRT}, while it keeps the functoriality properties
related to Dehn-filling.

\begin{remark}
\label{rem.super}
In fact, Conjecture~\ref{conj.1g} can be formulated for quantum knot polynomials
colored by irreducible representations of simple Lie superalgebras. We will discuss
this extension is a future publication.
\end{remark}

\subsection{Plan of the proof}
\label{sub.plan}

At first, Theorem~\ref{thm.1} feels a daunting task, since
the sought elements of the Habiro ring
$\calH_{\BZ[t^{\pm 1}, \Delta(t)^{-1}]/\BZ[t^{\pm 1}]}$ are collections 
of power series at roots of unity that $p$-adically glue, and in addition
they are topological invariants. However, we can achieve both tasks at the same time
by combining the following ideas. 

\noindent
$\bullet$
A planar projection $D$ of a knot $K$ (or even better, a braid presentation)
leads to a proper $q$-hypergeometric multisum for the colored Jones polynomial. The
shape of the multisum is dictated by the entries of the $R$-matrix, and the
summand is a product of $q$-binomials times $q$-Pochhammer symbols. 

\noindent
$\bullet$
The $q=1$ expansion of the above multisum and its deformations (inserting
independent variables at each crossing), analyzed by Rozansky, combined with the
MacMahon Master theorem and the Burau representation leads to a well-defined series
in $\BZ[t^{\pm 1}, \Delta(t)^{-1}][\![q-1]\!]$.

\noindent
$\bullet$
The Lagrange inversion formula gives further rationality results, as well as a
distinguished rational solution to the associated Nahm equations.

\noindent
$\bullet$
Expanding the $q$-binomials and the $q$-Pochhammer symbols at roots of unity,
and using the above rationality, one obtains collections of series around each
root of unity, following the ideas of~\cite{GSWZ} and~\cite{GW:explicit}.

\noindent
$\bullet$
The gluing of these series follows from the fact that they are expansions of
the same $q$-hypergeometric series.

\noindent
$\bullet$
This proves the existence of a lift $J_K(t,q)$ that satisfies~\eqref{eq1}. Since
elements of these Habiro rings are uniquely determined by their $q-1$-expansions,
the specialization~\eqref{eq1} implies the topological invariance of $J_K(t,q)$.

\subsection*{Acknowledgements} 

The authors wish to thank Peter Scholze, Ferdinand Wagner and Don Zagier for many
enlightening conversations. C.W. has been supported by the Huawei
Young Talents Program at Institut des Hautes Études Scientifiques, France. The
authors wish to thank the Max-Planck Institute for Mathematics and the Institut des
Hautes Études Scientifiques for their hospitality. 


\section{Quantum knot invariants and rationality}
\label{sec.cj}

All quantum knot polynomials (such as the colored Jones polynomials) are obtained
by contractions of tensors from a planar projection of a knot, or alternatively by
taking traces of representations of braid representatives of a knot. In this 
section, we discuss in detail the structure of these state-sum formulas of the
colored Jones polynomial of a knot. As it turns out, these state-sum formulas
lead to elements of the Habiro rings. 

\subsection{A state-sum for the colored Jones polynomial}
\label{sub.state}

The Jones polynomial invariant of links in $S^3$ was defined using representations
of the braid group which are determined by an $R$-matrix, i.e., an endomorphism
$R \in \End(V^{\otimes 2})$ that satisfies the Yang-Baxter equation
\be
R_1 R_2 R_1 = R_2 R_1 R_2 \in \End(V^{\otimes 3})
\ee
where $R_1=R \otimes \mathrm{id}_V$ and $R_2=\mathrm{id}_V \otimes R$.
Such a $R$-matrix, originally discovered by Jones~\cite{Jones} gives rise to
a compatible representation of the braid groups $B_r$ or $r \geq 2$ stands. 
The latter are generated by $\sigma_j$ for $j=1,\dots,r-1$ given by
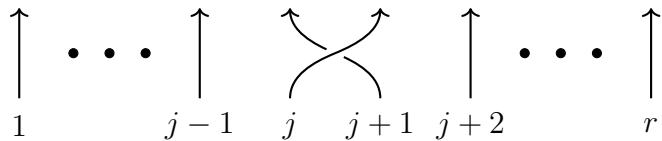
\begin{center}
\begin{tikzpicture}[scale=1.2]
\draw[thick,->,xshift=-1cm,yshift=0cm](0,0)--(0,1);
\filldraw(-0.4,0.5) circle (0.05cm);
\filldraw(0,0.5) circle (0.05cm);
\filldraw(0.4,0.5) circle (0.05cm);
\draw[thick,->,xshift=1cm,yshift=0cm](0,0)--(0,1);
\draw[thick,->,xshift=2cm,yshift=0cm] (1,0) to[out angle=90,
in angle=-90,curve through = {(0.5,0.5)}] (0,1);
\filldraw[white,xshift=2cm,yshift=0cm](0.5,0.5) circle (0.1cm);
\draw[thick,->,xshift=2cm,yshift=0cm](0,0) to[out angle=90,
in angle=-90, curve through= {(0.5,0.5)}] (1,1);
\draw[thick,->,xshift=4cm,yshift=0cm](0,0)--(0,1);
\filldraw[xshift=5cm](-0.4,0.5) circle (0.05cm);
\filldraw[xshift=5cm](0,0.5) circle (0.05cm);
\filldraw[xshift=5cm](0.4,0.5) circle (0.05cm);
\draw[thick,->,xshift=6cm,yshift=0cm](0,0)--(0,1);

\draw(-1,-0.3) node {$1$};
\draw(1,-0.3) node {$j-1$};
\draw(2,-0.3) node {$j$};
\draw(3,-0.3) node {$j+1$};
\draw(4,-0.3) node {$j+2$};
\draw(6,-0.3) node {$r$};
\end{tikzpicture}
\end{center}
that satisfy the relations
\be
\sigma_j\sigma_{j+1}\sigma_j \= \sigma_{j+1}\sigma_j\sigma_{j+1}
\ee
and $\sigma_i\sigma_j=\sigma_j\sigma_i$ if $|i-j|>1$. Then, an $R$-matrix gives
a sequence of homomorphisms
\be
\label{rhor}
\rho_r: B_r \to \End(V^{\otimes r}), \qquad \s_j \mapsto
\mathrm{id}_{V^{\otimes i-1}} \otimes R \otimes \mathrm{id}_{V^{\otimes r-i-1}} \,.
\ee

%

We will focus on the colored Jones polynomial of a knot, colored by the
$n$-dimensional irreducible representation $V$ of $\mathfrak{sl}_2$. Using the basis
$(e_0,e_1,\dots,e_{n-1})$ of $V$, the $R$-matrix is given by~\cite[Eqn.(2.13)]{Rozansky}
(see also~\cite[Cor.2.32]{Kirby-Melvin}\footnote{Where we substitute
  $(i,j,n,m,m')$ in their formula by $(i-k-n/2,j+k-n/2,k,n/2,n/2)$.})
\be
\begin{aligned}
  R^+(e_i\otimes e_j)
  &\= q^{n^2/4}\sum_{k=0}^\infty R^+_{ij,k}e_{j+k}\otimes e_{i-k}\\
  R^+_{ij,k}
  &\=
  (-1)^k
  q^{ij+k(k-1)/2}t^{-(i+j+k)/2}
  \bigg[\!\!\begin{array}{c} i\\k\end{array}\!\!\bigg]_q(q^{-j}t;q^{-1})_{k}\,,\\
  R^{-}(e_i\otimes e_j)
  &\=q^{-n^2/4}\sum_{k=0}^\infty R^{-}_{ij,k}e_{j-k}\otimes e_{i+k}\,,\\
  R^{-}_{ij,k}
  &\=
  q^{-(i+k)(j-k)}
  t^{(i+j-k)/2}
  \bigg[\!\!\begin{array}{c} j\\k\end{array}\!\!\bigg]_q(q^{-i}t;q^{-1})_{k}
\end{aligned}
\ee
where $t=q^n$ and throughout this paper we take
\be
  \binom{k}{\ell}
  \=
  \bigg[\!\!\begin{array}{c} k\\\ell\end{array}\!\!\bigg]_{q}
  \=
  0
\ee
whenever we have integers $k,\ell$ not satisfying $0\leq\ell\leq k$.

To compute the colored Jones polynomial $J_{K,n}(q)$ of a knot $K$, we represent the
knot as the closure of a braid $\beta \in B_r$. In fact, we will work with long knots,
obtained by the closure of all but the first strand of $\beta$. Then, the colored
Jones polynomial of $K$ is
\be\label{eq:state.sum}
J_{K,n}(q) \= \tr'((\mathrm{id}_V \otimes h^{\otimes r-1})\rho_r(\beta))
\ee
where $h \in \End(V)$ is a diagonal matrix that encodes the contributions of the
caps and $\tr'$ means that we take the trace on all but the first factor of $V$
in $V^{\otimes r}$. Doing so, the colored Jones polynomial
$J_{K,n}(q) \in \BZ[q^{\pm 1}]$ is normalized $1$ at the unknot. 

Now comes an important point. Because the closure of $\beta$ is a knot, the above sums
are terminating and in fact can be extended to the (infinite dimensional) dual
Verma module $W$ of highest weight $n-1$ with basis $(e_0,e_1,e_2,\dots)$.
The formula that one obtains for the colored Jones polynomial is best explained
by an example. 

\begin{example}
\label{ex.944a}
Consider the knot $9_{44}$ represented as the closure of the braid with four strands
\be
\label{b944}
\sigma_1^{-3} \sigma_2^{-1} \sigma_3 \sigma_1^2 \sigma_2^{-1} \sigma_3
\in B_4
\ee
This braid is pictured as follows (composing the braid generators from left to right
in the word, and from bottom to top on the braid):

\begin{center}
\begin{tikzpicture}[scale=1,rotate=-90]
\foreach \x in {1,2,3,4,5,6,7,8} \draw[yshift=\x cm](-0.2,0)--(3.2,0);
\draw[thick,->,xshift=0cm,yshift=1cm](0,0) to[out angle=90, in angle=-90,
curve through= {(0.5,0.5)}] (1,1);
\filldraw[white,xshift=0cm,yshift=1cm](0.5,0.5) circle (0.1cm);
\draw[thick,->,xshift=0cm,yshift=1cm] (1,0) to[out angle=90, in angle=-90,
curve through = {(0.5,0.5)}] (0,1);

\draw[thick,->,xshift=2cm,yshift=1cm](0,0)--(0,1);

\draw[thick,->,xshift=3cm,yshift=1cm](0,0)--(0,1);
\draw[thick,->,xshift=0cm,yshift=2cm](0,0) to[out angle=90, in angle=-90,
curve through= {(0.5,0.5)}] (1,1);
\filldraw[white,xshift=0cm,yshift=2cm](0.5,0.5) circle (0.1cm);
\draw[thick,->,xshift=0cm,yshift=2cm] (1,0) to[out angle=90, in angle=-90,
curve through = {(0.5,0.5)}] (0,1);

\draw[thick,->,xshift=2cm,yshift=2cm](0,0)--(0,1);

\draw[thick,->,xshift=3cm,yshift=2cm](0,0)--(0,1);
\draw[thick,->,xshift=0cm,yshift=3cm](0,0) to[out angle=90, in angle=-90,
curve through= {(0.5,0.5)}] (1,1);
\filldraw[white,xshift=0cm,yshift=3cm](0.5,0.5) circle (0.1cm);
\draw[thick,->,xshift=0cm,yshift=3cm] (1,0) to[out angle=90, in angle=-90,
curve through = {(0.5,0.5)}] (0,1);

\draw[thick,->,xshift=2cm,yshift=3cm](0,0)--(0,1);

\draw[thick,->,xshift=3cm,yshift=3cm](0,0)--(0,1);
\draw[thick,->,xshift=1cm,yshift=4cm](0,0) to[out angle=90, in angle=-90,
curve through= {(0.5,0.5)}] (1,1);
\filldraw[white,xshift=1cm,yshift=4cm](0.5,0.5) circle (0.1cm);
\draw[thick,->,xshift=1cm,yshift=4cm] (1,0) to[out angle=90, in angle=-90,
curve through = {(0.5,0.5)}] (0,1);

\draw[thick,->,xshift=0cm,yshift=4cm](0,0)--(0,1);

\draw[thick,->,xshift=3cm,yshift=4cm](0,0)--(0,1);
\draw[thick,->,xshift=0cm,yshift=5cm] (1,0) to[out angle=90, in angle=-90,
curve through = {(0.5,0.5)}] (0,1);
\filldraw[white,xshift=0cm,yshift=5cm](0.5,0.5) circle (0.1cm);
\draw[thick,->,xshift=0cm,yshift=5cm](0,0) to[out angle=90, in angle=-90,
curve through= {(0.5,0.5)}] (1,1);

\draw[thick,->,xshift=2cm,yshift=5cm] (1,0) to[out angle=90, in angle=-90,
curve through = {(0.5,0.5)}] (0,1);
\filldraw[white,xshift=2cm,yshift=5cm](0.5,0.5) circle (0.1cm);
\draw[thick,->,xshift=2cm,yshift=5cm](0,0) to[out angle=90, in angle=-90,
curve through= {(0.5,0.5)}] (1,1);
\draw[thick,->,xshift=0cm,yshift=6cm] (1,0) to[out angle=90, in angle=-90,
curve through = {(0.5,0.5)}] (0,1);
\filldraw[white,xshift=0cm,yshift=6cm](0.5,0.5) circle (0.1cm);
\draw[thick,->,xshift=0cm,yshift=6cm](0,0) to[out angle=90, in angle=-90,
curve through= {(0.5,0.5)}] (1,1);

\draw[thick,->,xshift=2cm,yshift=6cm](0,0)--(0,1);

\draw[thick,->,xshift=3cm,yshift=6cm](0,0)--(0,1);
\draw[thick,->,xshift=1cm,yshift=7cm](0,0) to[out angle=90, in angle=-90,
curve through= {(0.5,0.5)}] (1,1);
\filldraw[white,xshift=1cm,yshift=7cm](0.5,0.5) circle (0.1cm);
\draw[thick,->,xshift=1cm,yshift=7cm] (1,0) to[out angle=90, in angle=-90,
curve through = {(0.5,0.5)}] (0,1);

\draw[thick,->,xshift=0cm,yshift=7cm](0,0)--(0,1);

\draw[thick,->,xshift=3cm,yshift=7cm](0,0)--(0,1);
\draw[thick,->,xshift=2cm,yshift=8cm] (1,0) to[out angle=90, in angle=-90,
curve through = {(0.5,0.5)}] (0,1);
\filldraw[white,xshift=2cm,yshift=8cm](0.5,0.5) circle (0.1cm);
\draw[thick,->,xshift=2cm,yshift=8cm](0,0) to[out angle=90, in angle=-90,
curve through= {(0.5,0.5)}] (1,1);

\draw[thick,->,xshift=0cm,yshift=8cm](0,0)--(0,1);

\draw[thick,->,xshift=1cm,yshift=8cm](0,0)--(0,1);
\end{tikzpicture}
\end{center}

We can compute the quantum invariant by taking the trace.
This reduces to a sum over variables attached to each crossing.
This induces labeling on the edges since we leave the first edge fixed with
label $0$, we get Figure~\ref{fig:9_44}.
\begin{figure}
\begin{center}
\begin{tikzpicture}[scale=2]
\foreach \x in {2,3,4,5,6,7,8} \draw[yshift=\x cm,opacity=0.25](-0.2,0)--(3.2,0);

\draw[yshift=1 cm,opacity=0.25](-0.2,0)--(5.2,0);
\draw[yshift=9 cm,opacity=0.25](-0.2,0)--(5.2,0);
\draw[thick,->,xshift=0cm,yshift=1cm](0,0) to[out angle=90, in angle=-90,
curve through= {(0.5,0.5)}] (1,1);
\filldraw[white,xshift=0cm,yshift=1cm](0.5,0.5) circle (0.1cm) node[red,right]
{$\;\; k_1$};
\draw[thick,->,xshift=0cm,yshift=1cm] (1,0) to[out angle=90, in angle=-90,
curve through = {(0.5,0.5)}] (0,1);

\draw[thick,->,xshift=2cm,yshift=1cm](0,0)--(0,1);

\draw[thick,->,xshift=3cm,yshift=1cm](0,0)--(0,1);

\draw[thick,xshift=2.05cm,yshift=1cm,green](0,0)--(0,1);

\draw[thick,->,xshift=0cm,yshift=2cm](0,0) to[out angle=90, in angle=-90,
curve through= {(0.5,0.5)}] (1,1);
\filldraw[white,xshift=0cm,yshift=2cm](0.5,0.5) circle (0.1cm) node[red,right]
{$\;\; k_2$};
\draw[thick,->,xshift=0cm,yshift=2cm] (1,0) to[out angle=90, in angle=-90,
curve through = {(0.5,0.5)}] (0,1);

\draw[thick,->,xshift=2cm,yshift=2cm](0,0)--(0,1);

\draw[thick,->,xshift=3cm,yshift=2cm](0,0)--(0,1);

\draw[thick,xshift=2.05cm,yshift=2cm,green](0,0)--(0,1);

\draw[thick,->,xshift=0cm,yshift=3cm](0,0) to[out angle=90, in angle=-90,
curve through= {(0.5,0.5)}] (1,1);
\filldraw[white,xshift=0cm,yshift=3cm](0.5,0.5) circle (0.1cm) node[red,right]
{$\;\; k_3$};
\draw[thick,->,xshift=0cm,yshift=3cm] (1,0) to[out angle=90, in angle=-90,
curve through = {(0.5,0.5)}] (0,1);

\draw[thick,->,xshift=2cm,yshift=3cm](0,0)--(0,1);

\draw[thick,->,xshift=3cm,yshift=3cm](0,0)--(0,1);

\draw[thick,xshift=2.05cm,yshift=3cm,green](0,0)--(0,1);
\draw[thick,->,xshift=1cm,yshift=4cm](0,0) to[out angle=90, in angle=-90,
curve through= {(0.5,0.5)}] (1,1);
\filldraw[white,xshift=1cm,yshift=4cm](0.5,0.5) circle (0.1cm) node[red,right]
{$\;\; k_4$};
\draw[thick,->,xshift=1cm,yshift=4cm] (1,0) to[out angle=90, in angle=-90,
curve through = {(0.5,0.5)}] (0,1);

\draw[thick,->,xshift=0cm,yshift=4cm](0,0)--(0,1);

\draw[thick,->,xshift=3cm,yshift=4cm](0,0)--(0,1);

\draw[thick,->,xshift=2.05cm,yshift=4cm,green](0,0)--(0,0.5);
\draw[thick,xshift=2.05cm,yshift=4cm,green](0,0.5)--(0,1);
\draw[thick,->,xshift=0cm,yshift=5cm] (1,0) to[out angle=90, in angle=-90,
curve through = {(0.5,0.5)}] (0,1);
\filldraw[white,xshift=0cm,yshift=5cm](0.5,0.5) circle (0.1cm) node[red,right]
{$\;\; k_7$};
\draw[thick,->,xshift=0cm,yshift=5cm](0,0) to[out angle=90, in angle=-90,
curve through= {(0.5,0.5)}] (1,1);

\draw[thick,->,xshift=2cm,yshift=5cm] (1,0) to[out angle=90, in angle=-90,
curve through = {(0.5,0.5)}] (0,1);
\filldraw[white,xshift=2cm,yshift=5cm](0.5,0.5) circle (0.1cm) node[red,right]
{$\;\; k_5$};
\draw[thick,->,xshift=2cm,yshift=5cm](0,0) to[out angle=90, in angle=-90,
curve through= {(0.5,0.5)}] (1,1);

\draw[thick,xshift=2.05cm,yshift=5cm,green](0,0) to[out angle=90, in angle=-90,
curve through= {(0.5,0.5)}] (1,1);
\draw[thick,->,xshift=0cm,yshift=6cm] (1,0) to[out angle=90, in angle=-90,
curve through = {(0.5,0.5)}] (0,1);
\filldraw[white,xshift=0cm,yshift=6cm](0.5,0.5) circle (0.1cm) node[red,right]
{$\;\; k_8$};
\draw[thick,->,xshift=0cm,yshift=6cm](0,0) to[out angle=90, in angle=-90,
curve through= {(0.5,0.5)}] (1,1);

\draw[thick,->,xshift=2cm,yshift=6cm](0,0)--(0,1);

\draw[thick,->,xshift=3cm,yshift=6cm](0,0)--(0,1);

\draw[thick,xshift=3.05cm,yshift=6cm,green](0,0)--(0,1);
\draw[thick,->,xshift=1cm,yshift=7cm](0,0) to[out angle=90, in angle=-90,
curve through= {(0.5,0.5)}] (1,1);
\filldraw[white,xshift=1cm,yshift=7cm](0.5,0.5) circle (0.1cm) node[red,right]
{$\;\; k_9$};
\draw[thick,->,xshift=1cm,yshift=7cm] (1,0) to[out angle=90, in angle=-90,
curve through = {(0.5,0.5)}] (0,1);

\draw[thick,->,xshift=0cm,yshift=7cm](0,0)--(0,1);

\draw[thick,->,xshift=3cm,yshift=7cm](0,0)--(0,1);

\draw[thick,xshift=3.05cm,yshift=7cm,green](0,0)--(0,1);
\draw[thick,->,xshift=2cm,yshift=8cm] (1,0) to[out angle=90, in angle=-90,
curve through = {(0.5,0.5)}] (0,1);
\draw[thick,xshift=2.05cm,yshift=8cm,green] (1,0) to[out angle=90, in angle=-90,
curve through = {(0.5,0.5)}] (0,1);
\filldraw[white,xshift=2cm,yshift=8cm](0.5,0.5) circle (0.1cm) node[red,right]
{$\;\; k_6$};
\draw[thick,->,xshift=2cm,yshift=8cm](0,0) to[out angle=90, in angle=-90,
curve through= {(0.5,0.5)}] (1,1);

\draw[thick,->,xshift=0cm,yshift=8cm](0,0)--(0,1);

\draw[thick,->,xshift=1cm,yshift=8cm](0,0)--(0,1);
\draw[thick,->](3,9) to[out angle=90, in angle=90, curve through= {(3.5,9.5)}] (4,9);
\draw[thick,->](4,9)--(4,1);
\draw[thick,->](4,1) to[out angle=-90, in angle=-90, curve through= {(3.5,0.5)}] (3,1);

\draw[thick,->](2,9) to[out angle=90, in angle=90, curve through= {(3.25,9.75)}] (4.5,9);
\draw[thick,->](4.5,9)--(4.5,1);
\draw[thick,->](4.5,1) to[out angle=-90, in angle=-90,
curve through= {(3.25,0.25)}] (2,1);

\draw[thick,green](2.05,9) to[out angle=90, in angle=90, curve through= {(3.25,9.7)}]
(4.45,9);
\draw[thick,green](4.45,9)--(4.45,1);
\draw[thick,green](4.45,1) to[out angle=-90, in angle=-90,
curve through= {(3.25,0.3)}] (2.05,1);

\draw[thick,->](1,9) to[out angle=90, in angle=90, curve through= {(3,10)}] (5,9);
\draw[thick,->](5,9)--(5,1);
\draw[thick,->](5,1) to[out angle=-90, in angle=-90, curve through= {(3,0)}] (1,1);


\draw[thick,->](0,0)--(0,1);
\draw[thick,->](0,9)--(0,10);
\draw[blue](0.1,0.5) node {\tiny $0$};
\draw[blue](1,2.1) node[right] {\tiny $k_1$};
\draw[blue](-0.15,3.1) node[rotate=90] {\tiny $k_1-k_2$};
\draw[blue](1,4.1) node {\tiny $k_1-k_2+k_3$};
\draw[blue](2,5.1) node {\tiny $k_1-k_2+k_3+k_4$};
\draw[blue](3.15,7) node[rotate=-90] {\tiny $k_1-k_2+k_3+k_4-k_5$};
\draw[blue](4.6,5) node[rotate=-90] {\tiny $k_1-k_2+k_3+k_4-k_5+k_6$};
\draw[blue](1,4.9) node {\tiny $k_1-k_2+k_3-k_5+k_6$};
\draw[blue](-0.15,6) node[rotate=90] {\tiny $k_1-k_2+k_3-k_5+k_6+k_7$};
\draw[blue](1,7.1) node {\tiny $k_1-k_2+k_3-k_5$};
\draw[blue](1,6.9) node {\tiny $+k_6+k_7-k_8$};
\draw[blue](2,8.1) node {\tiny $k_1-k_2+k_3-k_5$};
\draw[blue](2,7.9) node {\tiny $+k_6+k_7-k_8+k_9$};
\draw[blue](4.1,5) node[rotate=-90] {\tiny $k_1-k_2+k_3-k_5+k_7-k_8+k_9$};
\draw[blue](2,6.5) node {\tiny $k_1-k_2+k_3+k_7-k_8+k_9$};
\draw[blue](5.1,5) node[rotate=-90] {\tiny $k_1-k_2+k_3+k_7-k_8$};
\draw[blue](-0.15,1.7) node[rotate=90] {\tiny $-k_2+k_3+k_7-k_8$};
\draw[blue](1,3.1) node {\tiny $k_3+k_7-k_8$};
\draw[blue](-0.15,4.2) node[rotate=90] {\tiny $k_7-k_8$};
\draw[blue](1,6.1) node {\tiny $-k_8$};
\draw[blue](0.1,9.5) node {\tiny $0$};
\end{tikzpicture}
\end{center}
\caption{Braid closure for $9_{44}$ with \textcolor{blue}{strand} and
  \textcolor{red}{crossing} variables labeled. It also depicts the
  \textcolor{green}{cycle} associated with the crossing labeled by $k_4$.}
\label{fig:9_44}
\end{figure}

The colored Jones polynomial $J_{9_{44},n}(q)$ is the evaluation at $t=q^{n-1}$
(a terminating sum for all $n \geq 1$) of the expression
\be
\label{eq:ex.quant.inv}
\begin{small}
\begin{aligned}
  &\!\!\!\sum_{k}(-1)^{k_5+k_6+k_7+k_8}q^{\frac{1}{2}Q(k)-3k_1 + 3k_2 - 3k_3
    - k_4 + \frac{3}{2}k_5 - \frac{3}{2}k_6 - \frac{5}{2}k_7
    + \frac{3}{2}k_8 - k_9}t^{-k_2-k_4+k_5-k_6-k_8-k_9+2}\\
&\!\!\!\bigg[\!\!\begin{array}{c} k_1-k_2+k_3+k_7-k_8\\k_1\end{array}\!\!\bigg]_{\!q}
\bigg[\!\!\begin{array}{c} k_1\\k_2\end{array}\!\!\bigg]_{\!q}\!
\bigg[\!\!\begin{array}{c} k_3+k_7-k_8\\k_3\end{array}\!\!\bigg]_{\!q}\\
&\!\!\!\bigg[\!\!\begin{array}{c} k_1-k_2+k_3+k_4-k_5+k_6\\
  k_4\end{array}\!\!\bigg]_{\!q}
\bigg[\!\!\begin{array}{c} k_1-k_2+k_3+k_4\\k_5\end{array}\!\!\bigg]_{\!q}
\bigg[\!\!\begin{array}{c} k_1-k_2+k_3-k_5+k_6+k_7-k_8+k_9\\
  k_6\end{array}\!\!\bigg]_{\!q}\\
&\!\!\!\bigg[\!\!\begin{array}{c} k_7-k_8\\k_7\end{array}\!\!\bigg]_{\!q}
\bigg[\!\!\begin{array}{c} k_1-k_2+k_3-k_5+k_6+k_7\\k_8\end{array}\!\!\bigg]_{\!q}
\bigg[\!\!\begin{array}{c} k_1-k_2+k_3+k_7-k_8+k_9\\k_9\end{array}\!\!\bigg]_{\!q}\\
&\!\!\!
(t;q^{-1})_{k_1}
(q^{k_2-k_3-k_7+k_8}t;q^{-1})_{k_2}
(q^{-k_1+k_2}t;q^{-1})_{k_3}(q^{-k_1+k_2-k_3}t;q^{-1})_{k_4}\\
&\!\!\!
(q^{-k_1+k_2-k_3+k_5-k_7+k_8-k_9}t;q^{-1})_{k_5}
(q^{-k_1+k_2-k_3-k_4+k_5}t;q^{-1})_{k_6}
(q^{-k_1+k_2-k_3+k_5-k_6}t;q^{-1})_{k_7}\\
&\!\!\!(q^{k_8}t;q^{-1})_{k_8}
(q^{-k_1+k_2-k_3+k_5-k_6-k_7+k_8}t;q^{-1})_{k_9}
\end{aligned}
\end{small}
\ee
associated to the braid~\eqref{b944} with $Q$ given below in Equation~\eqref{AB944}.
\end{example}
The above state-sum can be written in the form

\be
\label{ABt}
\begin{aligned}
&f_{A,B,\nu,\ve}(t,q)\\
&=
t^{\frac{r-1}{2}}\sum_{k\in\BZ_{\geq0}^N}
q^{\frac{1}{2}Q(k)-\nu(k)}
\prod_{i=1}^N(-\ve_iq^{\frac{-\ve_i-1}{4}}t^{-\frac{1}{2}})^{k_i}
\bigg[\!\!\begin{array}{c} (A^{t}k)_i \\
k_i\end{array}\!\!\bigg]_{q}
(q^{-(B^{t} k)_i}t;q^{-1})_{k_i}
t^{-\frac{\ve_i}{2} ((A^{t}k)_i + (B^{t}k)_i+1)},
\end{aligned} 
\ee
which can be viewed as an element $f_{A,B,\nu,\ve}(t,q) \in \BZ[q^{\pm 1}][\![t-1]\!]$, 
\be
\label{Jnf}
J_{K,n}(q) \= f_{A,B,\nu,\ve}(q^{n-1},q), \qquad (n >0),
\ee
and where
\begin{itemize}
\item
  $k_i$ is the variable associated to the $i$-crossing for $i=1,\dots,N$,
\item
  $\ve=(\ve_1,\dots,\ve_N)$ records the signs of the crossings,
  with $\ve_i=1$ (resp., $-1$) if $i$ is a positive (resp., negative) crossing,
\item
  $A$ and $B$ are square matrices with rows and columns indexed by the crossings,
  where the columns of $A$ store the linear forms in the binomial coefficients
  and the columns of $B$ store the linear forms that appear in
  the $q$-Pochhammer symbols,
\item
  $Q=Q_{A,B,\ve}$ is a quadratic form determined purely by $A,B,\ve$
  \be
  Q_{A,B,\ve}(k)
  \=
  \sum_{i:\ve_i=1}
  2(A^{t}k)_i(B^{t}k)_i+k_i^2
  -
  \sum_{i:\ve_i=-1}
  2((A^{t}k)_i-k_i)((B^{t}k)_i+k_i)\,,
  \ee
\item
  $\nu=\sum_{i=1}^{r-1}\nu_i$ is the sum of the linear forms in strands forming
  the closure of the braid.
\end{itemize}

In the above example, we have $N=9$ and

\be
\label{AB944}
\begin{tiny}
\begin{aligned}
A
\=
\begin{pmatrix}
  1 & 1 & 0 & 1 & 1 & 1 & 0 & 1 & 1\\
  -1 & 0 & 0 & -1 & -1 & -1 & 0 & -1 & -1\\
  1 & 0 & 1 & 1 & 1 & 1 & 0 & 1 & 1\\
  0 & 0 & 0 & 1 & 1 & 0 & 0 & 0 & 0\\
  0 & 0 & 0 & -1 & 0 & -1 & 0 & -1 & 0\\
  0 & 0 & 0 & 1 & 0 & 1 & 0 & 1 & 0\\
  1 & 0 & 1 & 0 & 0 & 1 & 1 & 1 & 1\\
  -1 & 0 & -1 & 0 & 0 & -1 & -1 & 0 & -1\\
  0 & 0 & 0 & 0 & 0 & 1 & 0 & 0 & 1
\end{pmatrix}, \quad
B \= \begin{pmatrix}
  0 & 0  & 1 & 1 & 1    &  1 &  1 &  1\\
  0 & -1 & -1 & -2 & -1 & -1 & -1 & -1\\
  0 & 1 & 0 & 1 & 1     &  1 &  1 &  1\\
  0 & 0 & 0 & 0 & 0     &  1 &  0 &  0\\
  0 & 0 & 0 & 0 & -1    & -1 & -1 & -1\\
  0 & 0 & 0 & 0 & 0     &  0 &  1 &  1\\
  0 & 1 & 0 & 0 & 1     &  0 &  0 &  1\\
  0 & -1 & 0 & 0 & -1   &  0 &  0 & -1\\
  0 & 0 & 0 & 0 & 1    &   0 &  0 & 0
\end{pmatrix},
\end{aligned}
\end{tiny}
\ee
and
\be
  Q\;=
  \begin{pmatrix}
    0&1&-2&1&-1&-1&-2&1&1\\
    1&0&1&-1&1&1&1&0&-1\\
    -2&1&0&1&-1&-1&0&-1&1\\
    1&-1&1&0&-1&0&2&-2&2\\
    -1&1&-1&-1&3&-1&-1&2&-1\\
    -1&1&-1&0&-1&1&0&-1&0\\
    -2&1&0&2&-1&0&-1&1&-1\\
    1&0&-1&-2&2&-1&1&-1&1\\
    1&-1&1&2&-1&0&-1&1&0
  \end{pmatrix},\quad
  \nu \;= \begin{pmatrix}
  3 \\ -3 \\ 3 \\ 1 \\ -2 \\ 1 \\ 2 \\ -2 \\ 1
  \end{pmatrix},\quad
  \ve \;= \begin{pmatrix}
  -1 \\ -1 \\ -1 \\ -1 \\ 1 \\ 1 \\ 1 \\ 1 \\ -1
  \end{pmatrix}.
\ee
We find that
\be
\begin{aligned}
J_{9_{44},1}(q) &\= 1\,,\\
J_{9_{44},2}(q) &\= q^{-2} - 2q^{-1} + 3 - 3q + 3q^2 - 2q^3 + 2q^4 - q^5\,,\\
J_{9_{44},3}(q) &\= -q^{-5} + 2q^{-4} + q^{-3} - 5q^{-2} + 4q^{-1} + 4
- 9q + 4q^2 + 6q^3 - 9q^4 + 2q^5 + 7q^6\\
&\qquad\quad - 7q^7 - q^8 + 7q^9 - 4q^{10} - 3q^{11} + 5q^{12} - q^{13}
- 2q^{14}\,,\\
J_{9_{44},4}(q) &\= -q^{-13} + q^{-12} + q^{-11} + 2q^{-10} - 4q^{-9}
- 4q^{-8} + 4q^{-7} + 8q^{-6} - 3q^{-5}\\
&\qquad\quad - 13q^{-4} + q^{-3} + 18q^{-2} + q^{-1} - 18 - 5q + 20q^2
+ 6q^3 - 18q^4 - 8q^5\\
&\qquad\quad + 17q^6 + 7q^7 - 13q^8 - 9q^9 + 11q^{10} + 8q^{11}
- 5q^{12} - 10q^{13} + 3q^{14} + 8q^{15}\\
&\qquad\quad + 3q^{16} - 8q^{17} - 6q^{18} + 5q^{19} + 8q^{20}
- 2q^{21} - 8q^{22} - q^{23} + 7q^{24} + 2q^{25}\\
&\qquad\quad - 4q^{26} - 2q^{27} + q^{28} + 2q^{29} - q^{30}\,.
\end{aligned}
\ee

\subsection{The $q\to 1$ limit, the Burau representation and the Alexander
  polynomial}
\label{sub.q=1}

Having presented the colored Jones polynomial in Equation~\eqref{Jnf} as the
evaluation of a series $f_{A,B,\nu,\ve}(t,q)$, we next discuss its $q=1$ limit 
$f_{A,B,\nu,\ve}(t,1) \in \BZ[\![t-1]\!]$. The key observation is the $q=1$
limit the $R$-matrices exists and is given by:
\be
\label{eq.R=1}
\begin{aligned}
R^+_{ij,k}|_{q=1}
\=
(-1)^k
t^{-(i+j+k)/2}
\binom{i}{k}(1-t)^{k}
\,,\quad
R^{-}_{ij,k}|_{q=1}
\=
t^{(i+j-k)/2}
\binom{j}{k}(1-t)^{k}\,.
\end{aligned}
\ee
This and Equation~\eqref{ABt} implies that
\be
\label{AB1}
\begin{aligned}
f_{A,B,\nu,\ve}(t,1)
&= t^{\frac{r-1-\sum_{i=1}^N \ve_i}{2}} 
\sum_{k\in\BZ_{\geq0}^N}
\prod_{i=1}^N(-\ve_i t^{-\frac{1}{2}})^{k_i}
\binom{(A^{t}k)_i}{k_i} 
(1-t)^{k_i}
t^{-\frac{\ve_i}{2} ((A^{t}k)_i + (B^{t}k)_i)} \,.
\end{aligned} 
\ee
We will see that $f_{A,B,\nu,\ve}(t,1)$ can be expressed as a sum over symmetric powers
of a single operator. To do so, we use the
additive identification $W=k[x]$ where $e_i \leftrightarrow x^i$ which 
induces an additive identification $W^{\otimes r}$ with $\BZ[x_1,\dots,x_r]$.
In particular, $e_0\otimes e_1=x_1$ and $e_1\otimes e_0=x_2$.

Consider the endomorphism $\calR^+ \in \End(\BZ[x_1,x_2])$ defined by
\be
\calR^+(x_1)\=y x_1+z x_2\,,\qquad
\calR^+(x_2)\=z x_1\,,\qquad
\calR^+
\=
\begin{pmatrix}
y & z\\
z & 0
\end{pmatrix}
\ee
for fixed $y,z$. 
Notice that $\mathrm{Sym}^{i+j}\BZ[x_1,x_2]$ has dimension $i+j+1$ and that
\be
\label{eq.calR+}
\mathrm{Sym}^{i+j}(\calR^+)(x_1^ix_2^{j})
\=
(yx_1+zx_2)^i(zx_1)^{j}
\=
\sum_k\binom{i}{k}y^kz^{i+j-k}x_1^{j+k}x_2^{i-k}\,.
\ee
Likewise we define the endomorphism $\calR^- \in \End(\BZ[x_1,x_2])$
\be
\calR^-(x_1)\=z^{-1} x_2\,,\qquad
\calR^-(x_2)\=z^{-1} x_1 + y x_2\,,\qquad
\calR^-
\=
\begin{pmatrix}
0 & z^{-1} \\
z^{-1} & y
\end{pmatrix}
\ee
for fixed $y,z$ and compute
\be
\label{eq.calR-}
\mathrm{Sym}^{i+j}(\calR^-)(x_1^ix_2^{j})
\=
(z^{-1}x_2)^{i}(z^{-1}x_1+yx_2)^j
\=
\sum_k\binom{j}{k}y^kz^{-i-j+k}x_1^{j-k}x_2^{i+k}\,.
\ee
Comparing Equation~\eqref{eq.R=1} with~\eqref{eq.calR+} and ~\eqref{eq.calR-},
we see that 
\be
\label{tsub}
\calR^\ve|_{y=(1-t^{-\ve}), z=t^{-1/2}} \= R^\ve|_{q=1}, \qquad \ve=+,- \,.
\ee  
It follows that, the $q \to 1$ limit $\rho_r|_{q=1}$ of the representations $\rho_r$
from~\eqref{rhor} exists
\be
\rho_r|_{q=1} : B_r \to \End(\BZ[x_1,\dots,x_r]) 
\ee
for all $r \geq 2$, and is equal to the Burau representation (see~\cite{Burau}
and the modern version~\cite{Birman}, up to a scalar renormalization)
\be
\fb_r: B_r \to \End(U_r), \qquad U_r=\BZ x_1 + \dots + \BZ x_r  
\ee
determined by
\be
\fb_2(\s_1) \=
\begin{pmatrix}
  1-t^{-1} & t^{-1/2} \\ t^{-1/2} & 0
\end{pmatrix}, \qquad \fb_2(\s_1^{-1})
\=
\begin{pmatrix}
  0 & t^{1/2} \\ t^{1/2} & 1-t
\end{pmatrix} \,.
\ee
Identifying $U_r$ with the degree 1 part of $\BZ[x_1,\dots,x_r]$ (where each
$x_i$ has degree 1), we obtain that $k[x_1,\dots,x_r]=S(U_r)$
where $S(W)$ is the symmetric algebra of $W$.

The key is the following lemma, which follows
from Equations~\eqref{eq.calR+}~and~\eqref{eq.calR-}.

\begin{lemma}
\label{lem.beta}
For all $r \geq 2$, we have
\be
\label{rhoslass}
\rho_r|_{q=1} \= S(\fb_r) \,. 
\ee
\end{lemma}

However, the state-sum formula~\eqref{AB1} is not related to the trace of the 
$S(\fb_r)$, but instead to the partial trace where we close all but the first
component of the braid; see Equation~\eqref{eq:state.sum}. To accommodate for this,
we need the following trimmed modification
\be
\label{trim}
S(\fb^\trim_r): \BZ[x_2,\dots,x_r] \subset \BZ[x_1,\dots,x_r] \stackrel{S(\fb_r)}{\to}
\BZ[x_1,\dots,x_r] \stackrel{x_1=0}{\to} \BZ[x_2,\dots,x_r] 
\ee
of the Burau representation. The $(r-1)\times (r-1)$ matrix of $\fb^\trim_r$ equals
to the matrix of $\fb^\trim_r$ obtained by removing the first row and column.
The above discussion implies the following.

\begin{lemma}
\label{lem.beta2}
For all $r \geq 2$, we have
\be
\label{rhoslass1}
f_{A,B,\nu,\ve}(t,1) \= t^{\tfrac{r-1-\text{writhe}}{2}} \tr\;S(\fb^\trim_r(\beta))
\ee
where $\text{writhe}=\sum_{i=1}^N \ve_i$ is the sum of the signs of the crossings of
$\beta$. 
Moreover, the matrix $\fb^\trim_r$ is nilpotent at $t=1$, hence the above sum
is a well-defined element of $\BZ[\![t-1]\!]$. 
\end{lemma}

The stated nilpotency follows from the fact that
the closure of the braid $\beta \in B_r$ is a knot. 
When $t=1$, $\fb_r$ is the permutation representation and, since $\beta$ is has closure a knot, $\fb_r(\beta)$ is an $r$-cycle.
This implies that for each $x_i$ there is some $k_i$ such that $\fb_r(\beta)^{k_i}(x_i)=x_1$.
Hence, the reduction $\fb^\trim_r$ at $t=1$ is nilpotent. 

The above lemma, together with the MacMahon Master theorem (discussed shortly)
implies the following.
\be
\label{fAB11}
f_{A,B,\nu,\ve}(t,1) \= \frac{t^{\tfrac{r-1-\text{writhe}}{2}}}{\det(I-\fb^\trim_r)} \,.
\ee
Finally, a theorem of Kauffman--Saleur~\cite[Eqn.(5.22)]{KS} identifies the
above determinant with the Alexander polynomial of a knot:
\be
\label{fAB1}
f_{A,B,\nu,\ve}(t,1) \= \frac{1}{\Delta_K(t)}\in\BZ[\![t-1]\!] \,,
\ee
As a consequence, this reproduces Rozansky's proof (in the case of $\fsl_2$)
of the Melvin--Morton--Rozansky Conjecture (theorem of~\cite{BG:MMR})
\be
J_{K,n}(q) \= \frac{1}{\Delta_K(t)} + O(q-1) \,.
\ee

\begin{remark}
\label{rem.burau}
The Burau representation has a trimmed form, discussed above, but also a reduced
form $\fb^{\mathrm{red}}_r$, and the three are related by
\be
\label{3burau}
\fb^\trim_r = (0 | I_{r-1}) \fb_r \begin{pmatrix} 0 \\ I_{r-1} \end{pmatrix},
\qquad
\fb^{\mathrm{red}}_r = (v | I_{r-1}) \fb_r \begin{pmatrix} 0 \\ I_{r-1} \end{pmatrix},
\,\,
v=(1,t^{1/2},\dots,t^{(r-1)/2})^t
\ee
Both the trimmed and the reduced form give a formula for the Alexander polynomial
\be
\label{A1}
\det(I_{r-1}-\fb^\trim_r)  \= t^{\tfrac{r-1-\text{writhe}}{2}} \Delta(t)
\ee
as shown by Kauffman--Saleur~\cite[Eqn.(5.22)]{KS}
and
\be
\label{A2}  
\det(I_{r-1}-\fb^{\mathrm{red}}_r) \= t^{\tfrac{-r+1-\text{writhe}}{2}}
\frac{1-t^r}{1-t}\Delta(t) 
\ee
shown by Kassel--Turaev~\cite[Sec.3.4]{KS}. A further formula for the Alexander
polynomial in terms of the Burau representation is discussed in the appendix, along with another proof Equation~\eqref{fAB1}. 
\end{remark}

\subsection{Rationality via the MacMahon's Master Theorem}
\label{sec:MacMahon}

In this section, we recall MacMahon's Master Theorem~\cite{Macmahon}, which
implies the rationality statement~\eqref{fAB1}. 

\begin{theorem}[MacMahon's Master theorem]
If $C\in M_{n\times n}(\BZ[C_{ij}])$ then
\be
  \frac{1}{\det(\mathrm{I}-C)}
  \=
  \sum_{\ell=0}^\infty
  \mathrm{Tr}(\mathrm{Sym}^\ell C)
  \=
  \frac{1}{(2\pi i)^n}
  \sum_{k\in\BZ_{\geq0}^{n}}
  \oint\cdots\oint
  \prod_{i=1}^{n}(Cz)_i^{k_i}
  \frac{dz_1\cdots dz_{n}}{z_1^{k_1+1}\cdots z_n^{k_n+1}}
  \;\in\;
  \BC[\![C_{ij}]\!]\,.
\ee
\end{theorem}

\begin{proof}
The operators $\delta_{jm}-\sum_{\ell=1}^n(\mathrm{I}-C)_{\ell j}\partial_{C_{\ell m}}$
annihilate the LHS and the RHS in the theorem and both sides satisfy the initial
conditions $1+\mathrm{O}(C_{ij})$. These operators have a unique power series
solution with this initial condition since letting
$\Delta_{\ell m}=\partial_{C_{\ell m}}$, the equations have the form
\be
\Delta f \= (\mathrm{I}-C)^{-t}f
\ee
and $\Delta\det(\mathrm{I}-C)f=0$, which implies that $\det(\mathrm{I}-C)f$ is
constant.
\end{proof}

The proof of Lemma~\ref{lem.beta} reveals that the $q\to1$ limit 
$f_{A,B,\nu,\ve}(t,1)$, which is assembled out of Burau matrices at each crossing
of the braid, admits a deformation using matrices $\calR^\pm$ with arbitrary
parameters at each crossing. This is best explained by our running
example~\ref{ex.944a} of the knot $9_{44}$. 

\begin{example}
\label{ex.944b}  
We will see this continuing the example~\eqref{ex.944a}. 
%
%
The endomorphism of $U_4$ is a $4 \times 4$ matrix
given by the product coming from the braid word~\eqref{b944} as follows:
\be
\begin{tiny}
\begin{aligned}
&\begin{pmatrix}
  0 & z^{-1} & 0 & 0\\
  z^{-1} & y_1 & 0 & 0\\
  0 & 0 & 1 & 0\\
  0 & 0 & 0 & 1
\end{pmatrix}
\begin{pmatrix}
  0 & z^{-1} & 0 & 0\\
  z^{-1} & y_2 & 0 & 0\\
  0 & 0 & 1 & 0\\
  0 & 0 & 0 & 1
\end{pmatrix}
\begin{pmatrix}
  0 & z^{-1} & 0 & 0\\
  z^{-1} & y_3 & 0 & 0\\
  0 & 0 & 1 & 0\\
  0 & 0 & 0 & 1
\end{pmatrix}
\begin{pmatrix}
  1 & 0 & 0 & 0\\
  0 & 0 & z^{-1} & 0\\
  0 & z^{-1} & y_4 & 0\\
  0 & 0 & 0 & 1
\end{pmatrix}
\\
&\qquad\times
\begin{pmatrix}
  1 & 0 & 0 & 0\\
  0 & 1 & 0 & 0\\
  0 & 0 & y_5 & z\\
  0 & 0 & z & 0
\end{pmatrix}
\begin{pmatrix}
  y_7 & z & 0 & 0\\
  z & 0 & 0 & 0\\
  0 & 0 & 1 & 0\\
  0 & 0 & 0 & 1
\end{pmatrix}
\begin{pmatrix}
  y_8 & z & 0 & 0\\
  z & 0 & 0 & 0\\
  0 & 0 & 1 & 0\\
  0 & 0 & 0 & 1
\end{pmatrix}
\begin{pmatrix}
  1 & 0 & 0 & 0\\
  0 & 0 & z^{-1} & 0\\
  0 & z^{-1} & y_9 & 0\\
  0 & 0 & 0 & 1
\end{pmatrix}
\begin{pmatrix}
  1 & 0 & 0 & 0\\
  0 & 1 & 0 & 0\\
  0 & 0 & y_6 & z\\
  0 & 0 & z & 0
\end{pmatrix}\\
&=\!\left(\!\!\begin{array}{c|c|c|c}
z^{-2}y_2y_7y_8 + y_2&
(z^{-1}y_1y_2 + z^{-3})y_7y_8 + (zy_1y_2 + z^{-1})&
y_8&
0\\
\hline
(z^{-3}y_2y_3 + z^{-5})y_5&
((z^{-2}y_1y_2 + z^{-4})y_3 + z^{-4}y_1)y_5&
z^{-1}y_4y_5&
1\\
\hline
\begin{array}{c}
\!\!\!\!((z^{-2}y_2y_3 + z^{-4})y_5y_9 + z^{-2}y_2y_7)y_6\!\!\!\!\\
 + (y_2y_3 + z^{-2})
\end{array}&
\begin{array}{c}
(((z^{-1}y_1y_2 + z^{-3})y_3 + z^{-3}y_1)y_5y_9\\
\!\!\!\!\!\!+(z^{-1}y_1y_2 + z^{-3})y_7)y_6
+ ((zy_1y_2 + z^{-1})y_3 + z^{-1}y_1)\!\!\!\!\!\!
\end{array}&
\begin{array}{c}
\!\!\!\!(y_4y_5y_9 + 1)y_6\!\!\!\!\\
+ z^2y_4
\end{array}&
zy_9y_6\\
\hline
(z^{-1}y_2y_3 + z^{-3})y_5y_9 + z^{-1}y_2y_7&
((y_1y_2 + z^{-2})y_3 + z^{-2}y_1)y_5y_9 + (y_1y_2 + z^{-2})y_7&
zy_4y_5y_9 + z&
z^2y_9
\end{array}\!\!\right)^{\!\!\!t}\!\!.
\end{aligned}
\end{tiny}
\ee
Therefore, the reduced operator
is given by the bottom $3\times 3$ matrix. 
Specializing $z=1$, we find that:
\be
\label{eq.944.q=1}
\begin{small}
\begin{aligned}
  &f_A(y) :=\sum_{k\in\BZ^{9}_{\geq0}}
  \binom{k_1-k_2+k_3+k_7-k_8}{k_1}
  \binom{k_1}{k_2}
  \binom{k_3+k_7-k_8}{k_3}
  \binom{k_1-k_2+k_3+k_4-k_5+k_6}{k_4}\\
  &\qquad\times\binom{k_1-k_2+k_3+k_4}{k_5}
  \binom{k_1-k_2+k_3-k_5+k_6+k_7-k_8+k_9}{k_6}
  \binom{k_7-k_8}{k_7}\\
  &\qquad\times\binom{k_1-k_2+k_3-k_5+k_6+k_7}{k_8}
  \binom{k_1-k_2+k_3+k_7-k_8+k_9}{k_9}y^k\\
  &\=
  \det\!\!
  \left(\!\!\begin{array}{c|c|c}
1-((y_1y_2 + 1)y_3 + y_1)y_5&
-y_4y_5&
-1\\
\hline
\begin{array}{c}
\!\!-(((y_1y_2 + 1)y_3 + y_1)y_5y_9 + (y_1y_2 + 1)y_7)y_6\!\!\\
- ((y_1y_2 + 1)y_3 + y_1)
\end{array}&
1-(y_4y_5y_9 + 1)y_6 - y_4&
-y_9y_6\\
\hline
-((y_1y_2 + 1)y_3 + y_1)y_5y_9 - (y_1y_2 + 1)y_7&
-y_4y_5y_9 - 1&
1-y_9 
\end{array}\!\!\right)^{\!\!\!\!-1}\\
&\= (1 - y_1 - y_3 - y_1 y_2 y_3 - y_4 - y_1 y_5 - y_3 y_5 - y_1 y_2 y_3 y_5 - y_6 + 
y_1 y_5 y_6 + y_3 y_5 y_6 + y_1 y_2 y_3 y_5 y_6 \\ & \hspace{0.6cm}
- y_7 - y_1 y_2 y_7 + y_4 y_7 + 
 y_1 y_2 y_4 y_7 - y_4 y_5 y_6 y_7 - y_1 y_2 y_4 y_5 y_6 y_7 - y_9 + y_4 y_9 - 
 y_4 y_5 y_6 y_9)^{-1}
\end{aligned}
\end{small}
\ee
On the other hand specializing $y_i=1-t$ for $i=1,2,3,4,9$, $y_i=1-t^{-1}$
for $i=5,6,7,8$, and $z=t^{-1/2}$, 
we find that $f_{A,B,\nu,\ve}(t,1) \in \BZ[\![t-1]\!]$ equals to
\be
\begin{aligned}
  &\sum_{k\in\BZ^{9}_{\geq0}}
  \binom{k_1-k_2+k_3+k_7-k_8}{k_1}
  \binom{k_1}{k_2}
  \binom{k_3+k_7-k_8}{k_3}
  \binom{k_1-k_2+k_3+k_4-k_5+k_6}{k_4}\\
  &\qquad\times\binom{k_1-k_2+k_3+k_4}{k_5}
  \binom{k_1-k_2+k_3-k_5+k_6+k_7-k_8+k_9}{k_6}
  \binom{k_7-k_8}{k_7}\\
  &\qquad\times\binom{k_1-k_2+k_3-k_5+k_6+k_7}{k_8}
  \binom{k_1-k_2+k_3+k_7-k_8+k_9}{k_9}\\
  &t^{-k_2-k_4+2k_5+k_7-k_9}(1-t)^{k_1+k_2+k_3+k_4+k_9}(1-t^{-1})^{k_5+k_6+k_7+k_8}\\
  &\=\frac{1}{t^4-4t^3+7t^2-4t+1}\\
  &\=
  1 - 2(t-1) + 3(t-1)^2 - 4(t-1)^3 + 4(t-1)^4 - 2(t-1)^5 - 3(t-1)^6\\
  &\qquad + 12(t-1)^7 - 25(t-1)^8 + 40(t-1)^9+\mathrm{O}(t-1)^{10}
  \,.
\end{aligned}
\ee
This agrees with the inverse of the Alexander polynomial of $9_{44}$.
\end{example}

For a general braid $\beta \in B_r$ with closure a knot, the deformation of
$f_{A,B,\nu,\ve}(t,1)$ that we illustrated above, depends only on the matrix $A$ and
with the notation of Equation~\eqref{ABt}, it is given by
\be
\label{fAy}
f_{A}(y)
\=
\sum_{k\in\BZ_{\geq0}^N}
\prod_{i=1}^N \binom{(A^{t}k)_i}{k_i} y_i^{k_i} \,.
\ee

The relation between $f_{A,B,\nu,\ve}(t,1)$ and $f_A(y)$ can be described by
analysing certain cycles on the knot.
To each crossing, we can associate a cycle that follows along the knot starting
from the first time it enters the crossing and ending when it encounters the
crossing for the second time.
Let $\gamma_i$ be the cycle associated to the $i$-th crossing.
(The cycle $\gamma_4$ associated to the crossing labeled by $k_4$ is pictured in
Figure~\ref{fig:9_44}.)
Let $s(i)=\pm 1$ be positive if $i$ first appears as an under-crossing (this is
the sign of $k_i$ in the binomial sums).
Finally, counting the crossing that $\gamma$ passes with multiplicity, let
$d_i=\text{\#(negative crossings in the cycle
  $\gamma_i$)}-\text{\#(positive crossings in the cycle
  $\gamma_i$)}-\ve_i(1+s(i))$.

\begin{lemma}
\label{lem.yt}
We have the following:
\be
\label{yt}
f_{A,B,\nu,\ve}(t,1)
\=
f_A(y)|_{
y_i\to(t^{\ve_i}-1)t^{s(i)d_i/2}}\,.
\ee
\end{lemma}
The proof is immediate from Equations~\eqref{eq.R=1}.
\begin{example}
  
For $9_{44}$, we have that $N=9$ and 
\be
\begin{aligned}
  s&\=(\phantom{-}1,-1,\phantom{-}1,\phantom{-}1,-1,\phantom{-}1,
  \phantom{-}1,-1,\phantom{-}1)\,,\\
  d&
  \=(\phantom{-}2,\phantom{-}0,\phantom{-}2,\phantom{-}0,-2,-2,
  \phantom{-}0,\phantom{-}2,\phantom{-}0)\,.
\end{aligned}
\ee
\end{example}

\subsection{Further rationality via Lagrange inversion}
\label{sub.lag}

In computing the expansion of the state-sums $f_{A,B,\nu,\ve}(t,q)$ at roots of unity,
we need rationality of further series, namely those obtained by replacing
$\binom{(A^tk)_i}{k_i}$ in~\eqref{fAy} by $\binom{s_i+(A^tk)_i}{k_i}$ for integers
$s_i$.

It turns out that these new sums are determined by the solutions of the Nahm
equation associated to the $A$-matrix, which comes from an application
of Lagrange inversion used by Rodriguez-Villegas~\cite{RV:Apoly}.

\begin{theorem}
\label{thm:lag.inv}
The system of equations
\be\label{eq:nahm}
  1-z_i\=-y_i\prod_jz_j^{A_{ij}}\,\qquad i=1,\dots,N\,,
\ee
has a unique solution in $\BQ[\![y_1,\dots,y_N]\!]$
such that
\be
  \prod_i z_i^{s_i}
  \=
  \delta(y)
  \sum_{k\in\BZ_{\geq0}^N}
  \prod_{i=1}^N
  \binom{s_i+ (A^t k)_i}{k_i}y_i^{k_i}\,.
\ee
where
\be
\label{deltay}
  \delta(y)
  \=
  \det(\mathrm{I}+\mathrm{diag}(1-z_i)\,A\,\diag(z_i^{-1}))\,.
\ee
\end{theorem}

Note that when $s_i=0$ the binomial sum coincides with $f_A(y)$ from
Equation~\eqref{fAy}.

The next example gives an unexpected source of matrices $A$ with integer entries
whose Nahm equations are rational (as opposed to algebraic) functions. Its proof
uses an argument suggested to us by Peter Scholze.

\begin{theorem}
\label{thm.zrat}
If $\delta(y)\in\BQ(y)$, then $z_i\in\BQ(y)$ for $i=1,\dots, N$, and conversely.
\end{theorem}

\begin{proof}
Notice that
\be
\begin{aligned}
  z_i^{-1}
  &\=
  \delta(y)
  \sum_{k\in\BZ_{\geq0}^N}
  \prod_{\ell=1}^N
  \binom{-\delta_{i\ell}+ (A^{t} k)_\ell}{k_\ell}y_\ell^{k_\ell}\,,\\
  &\=
  \delta(y)
  \sum_{k\in\BZ_{\geq0}^N}
  \frac{-k_i+(A^{t} k)_i}{(Ak)_\ell}
  \prod_{\ell=1}^N
  \binom{\sum_{j=1}^N (A^{t} k)_\ell}{k_\ell}y_\ell^{k_\ell}\,.
\end{aligned}
\ee
We see that for $\th_i=y_i\partial_{y_i}$, we have
\be
  \Big(\sum_{j=1}^NA_{ji}\th_j\Big)\frac{z_i^{-1}}{\delta(y)}
  \=
  \Big(-\th_i+\sum_{j=1}^NA_{ji}\th_j\Big)\frac{1}{\delta(y)}\in\BQ(y)\,.
\ee
Moreover, we see that $z_i^{-1}|_{y_i=0}\in\BQ(y)$.

Consider this equation in the coordinates $\log(y_j)$.
If $A_{ii}\neq0$, then we can complete $\log(\eta_1)=\log(y_i)$ to a set of
coordinates $\log(\eta_j)$, where $\eta_j$ are monomials in $y_j^{\pm1/D}$ for
some $D$ (and $y_j$ are monomials in $\eta_j^{\pm1}$) such that
$\sum_{j=1}^NA_{ji}\th_j=\sum_{j=1}^NA_{ji}\partial_{\log(y_j)}
=\partial_{\log(\eta_1)}$ and $y_i$ only appears in $\eta_j$ with a positive sign.
Therefore, we see that $\eta_1\partial_{\eta_1}\delta^{-1}z_i\in\BQ(\eta)$ for
some $d$.

Notice that the $z_i$ solve a system of algebraic equations over $\BQ(\eta)$ and
hence there is a finite order Galois group $G$ of automorphisms acting on $z_i$.
We see that
\be
  Z_i^{-1}
  \=
  \frac{1}{|G|}\sum_{\sigma\in G}\sigma\cdot z_i^{-1}\in\BQ(\eta)\,.
\ee
This all implies that
\be
  \partial_{\eta_1}\frac{Z_i^{-1}-z_i^{-1}}{\delta}
  \=0\=
  \frac{Z_i^{-1}-z_i^{-1}}{\delta}\Big|_{\eta_1=0}\,,
\ee
which in turn shows that $Z_i^{-1}=z_i^{-1}\in\BQ(\eta)$.
Moreover, since $z_i^{-1}\in\BQ[\![y]\!]$, we conclude that $z_i^{-1}\in\BQ(y)$.

Finally, if $A_{ii}=0$, then apply the same argument for $y_iz_i$ as opposed
to $z_i^{-1}$, which will work when $A_{ii}\neq1$.

The converse is obvious from Equation~\eqref{deltay}.
\end{proof}

We illustrate the previous theorem with our running example~\ref{ex.944a}.

\begin{example}
\label{ex.944c}
Continuing Example~\ref{ex.944a}, the Nahm equations~\eqref{eq:nahm} for the
matrix $A$ of~\eqref{AB944} have the rational solution
\be
\fontsize{4}{12}\selectfont
\begin{aligned}
&z_1 =\\
& (-1 - y_1 y_2 - y_4 - y_1 y_2 y_4 - y_6 - y_1 y_2 y_6 - y_1 y_8 - y_3 y_8 - 
y_1 y_2 y_3 y_8 - y_7 y_8 - y_1 y_2 y_7 y_8 - y_4 y_7 y_8 - y_1 y_2 y_4 y_7 y_8\\
& - 
y_9 - y_1 y_2 y_9 - y_4 y_9 - y_1 y_2 y_4 y_9 + y_4 y_5 y_6 y_9 + 
y_1 y_2 y_4 y_5 y_6 y_9 - y_1 y_8 y_9 - y_3 y_8 y_9 - y_1 y_2 y_3 y_8 y_9 + 
y_1 y_5 y_6 y_8 y_9\\
& + y_3 y_5 y_6 y_8 y_9 + y_1 y_2 y_3 y_5 y_6 y_8 y_9 - 
y_7 y_8 y_9 - y_1 y_2 y_7 y_8 y_9 - y_4 y_7 y_8 y_9 - y_1 y_2 y_4 y_7 y_8 y_9 + 
y_4 y_5 y_6 y_7 y_8 y_9 + y_1 y_2 y_4 y_5 y_6 y_7 y_8 y_9)\\
\hline
&(-1 - y_1 - y_4 + 
y_1 y_5 - y_6 + y_1 y_5 y_6 - y_3 y_8 - y_7 y_8 - y_4 y_7 y_8 + y_1 y_5 y_7 y_8\\
& - 
y_9 - y_4 y_9 + y_4 y_5 y_6 y_9 - y_3 y_8 y_9 + y_3 y_5 y_6 y_8 y_9 - y_7 y_8 y_9 -
y_4 y_7 y_8 y_9 + y_4 y_5 y_6 y_7 y_8 y_9)
\\
&z_2 =\\
& -(1 - y_2 - y_2 y_4 - y_5 - y_2 y_6 - y_5 y_6 - y_8 - y_2 y_3 y_8 - 
  y_2 y_7 y_8 - y_2 y_4 y_7 y_8 - y_5 y_7 y_8 - y_2 y_9 - y_2 y_4 y_9 + 
  y_2 y_4 y_5 y_6 y_9 - y_8 y_9\\
& - y_2 y_3 y_8 y_9 + y_5 y_6 y_8 y_9 + 
  y_2 y_3 y_5 y_6 y_8 y_9 - y_2 y_7 y_8 y_9 - y_2 y_4 y_7 y_8 y_9 + 
  y_2 y_4 y_5 y_6 y_7 y_8 y_9)\\
  \hline
  &(-1 - y_1 y_2 + y_5 + y_1 y_2 y_5 + y_5 y_6 + 
  y_1 y_2 y_5 y_6 + y_8 + y_5 y_7 y_8 + y_1 y_2 y_5 y_7 y_8 + y_8 y_9 - 
  y_5 y_6 y_8 y_9)
  \\
&z_3 =\\
& (-1 - y_1 - y_4 + y_1 y_5 - y_6 + y_1 y_5 y_6 - y_3 y_8 - y_7 y_8 - 
y_4 y_7 y_8 + y_1 y_5 y_7 y_8 - y_9 - y_4 y_9 + y_4 y_5 y_6 y_9 - y_3 y_8 y_9\\
& + 
y_3 y_5 y_6 y_8 y_9 - y_7 y_8 y_9 - y_4 y_7 y_8 y_9 + 
y_4 y_5 y_6 y_7 y_8 y_9)\\
\hline
&(-1 - y_1 - y_3 - y_1 y_2 y_3 - y_4 + y_1 y_5 + 
y_3 y_5 + y_1 y_2 y_3 y_5 - y_6 + y_1 y_5 y_6 + y_3 y_5 y_6 + y_1 y_2 y_3 y_5 y_6\\
& -
 y_7 y_8 - y_4 y_7 y_8 + y_1 y_5 y_7 y_8 + y_3 y_5 y_7 y_8 + 
y_1 y_2 y_3 y_5 y_7 y_8 - y_9 - y_4 y_9 + y_4 y_5 y_6 y_9 - y_7 y_8 y_9 - 
y_4 y_7 y_8 y_9 + y_4 y_5 y_6 y_7 y_8 y_9)
\\
&z_4 =\\
& -(
   1 + y_1 y_2 + y_4 y_5 + y_1 y_2 y_4 y_5 - y_8 + y_1 y_5 y_8 + y_3 y_5 y_8 + 
    y_1 y_2 y_3 y_5 y_8 + y_4 y_5 y_7 y_8 + y_1 y_2 y_4 y_5 y_7 y_8 - y_8 y_9)\\
    \hline
&(-1 - 
    y_1 y_2 - y_4 - y_1 y_2 y_4 + y_4 y_5 y_6 + y_1 y_2 y_4 y_5 y_6 + y_8 + y_4 y_8 - 
    y_1 y_5 y_8 - y_3 y_5 y_8 - y_1 y_2 y_3 y_5 y_8 + y_8 y_9 + y_4 y_8 y_9 - 
    y_4 y_5 y_6 y_8 y_9)
    \\
&z_5 =\\
& -(-1 - y_1 y_2 + y_5 + y_1 y_2 y_5 + y_5 y_6 + y_1 y_2 y_5 y_6 + y_8 + 
    y_5 y_7 y_8 + y_1 y_2 y_5 y_7 y_8 + y_8 y_9 - y_5 y_6 y_8 y_9)\\
    \hline
&(
   1 + y_1 y_2 + y_4 y_5 + y_1 y_2 y_4 y_5 - y_8 + y_1 y_5 y_8 + y_3 y_5 y_8 + 
   y_1 y_2 y_3 y_5 y_8 + y_4 y_5 y_7 y_8 + y_1 y_2 y_4 y_5 y_7 y_8 - y_8 y_9)
   \\
&z_6 =\\
& -(-1 - y_1 y_2 - y_4 - y_1 y_2 y_4 + y_4 y_5 y_6 + y_1 y_2 y_4 y_5 y_6 - 
  y_1 y_8 - y_3 y_8 - y_1 y_2 y_3 y_8 - y_6 y_8 + y_1 y_5 y_6 y_8\\
& + 
  y_3 y_5 y_6 y_8 + y_1 y_2 y_3 y_5 y_6 y_8 - y_7 y_8 - y_1 y_2 y_7 y_8 - 
  y_4 y_7 y_8 - y_1 y_2 y_4 y_7 y_8 + y_4 y_5 y_6 y_7 y_8 + 
  y_1 y_2 y_4 y_5 y_6 y_7 y_8)\\
  \hline
&(1 + y_1 y_2 + y_4 + y_1 y_2 y_4 + y_6 + 
  y_1 y_2 y_6 + y_1 y_8 + y_3 y_8 + y_1 y_2 y_3 y_8 + y_7 y_8 + y_1 y_2 y_7 y_8 + 
  y_4 y_7 y_8 + y_1 y_2 y_4 y_7 y_8 - y_6 y_8 y_9)
  \\
&z_7 =\\
& (-1 - y_1 - y_3 - y_1 y_2 y_3 - y_4 + y_1 y_5 + y_3 y_5 + y_1 y_2 y_3 y_5 - 
y_6 + y_1 y_5 y_6 + y_3 y_5 y_6 + y_1 y_2 y_3 y_5 y_6 - y_7 y_8 - y_4 y_7 y_8\\
& + 
y_1 y_5 y_7 y_8 + y_3 y_5 y_7 y_8 + y_1 y_2 y_3 y_5 y_7 y_8 - y_9 - y_4 y_9 + 
y_4 y_5 y_6 y_9 - y_7 y_8 y_9 - y_4 y_7 y_8 y_9 + y_4 y_5 y_6 y_7 y_8 y_9)\\
\hline
&(-1 - 
y_1 - y_3 - y_1 y_2 y_3 - y_4 + y_1 y_5 + y_3 y_5 + y_1 y_2 y_3 y_5 - y_6 + 
y_1 y_5 y_6 + y_3 y_5 y_6 +\\
& y_1 y_2 y_3 y_5 y_6 - y_7 - y_1 y_2 y_7 - y_4 y_7 - 
y_1 y_2 y_4 y_7 + y_4 y_5 y_6 y_7 + y_1 y_2 y_4 y_5 y_6 y_7 - y_9 - y_4 y_9 + 
y_4 y_5 y_6 y_9)
\\
&z_8 =\\
& (-1 - y_1 y_2 - y_4 - y_1 y_2 y_4 + y_4 y_5 y_6 + y_1 y_2 y_4 y_5 y_6 + y_8 + 
y_4 y_8 - y_1 y_5 y_8 - y_3 y_5 y_8 - y_1 y_2 y_3 y_5 y_8 + y_8 y_9 + y_4 y_8 y_9 -
 y_4 y_5 y_6 y_8 y_9)\\
 \hline
&(-1 - y_1 y_2 - y_4 - y_1 y_2 y_4 + y_4 y_5 y_6 + 
y_1 y_2 y_4 y_5 y_6 - y_1 y_8 - y_3 y_8 - y_1 y_2 y_3 y_8 - y_6 y_8 + 
y_1 y_5 y_6 y_8\\
& + y_3 y_5 y_6 y_8 + y_1 y_2 y_3 y_5 y_6 y_8 - y_7 y_8 - 
y_1 y_2 y_7 y_8 - y_4 y_7 y_8 - y_1 y_2 y_4 y_7 y_8 + y_4 y_5 y_6 y_7 y_8 + 
y_1 y_2 y_4 y_5 y_6 y_7 y_8)
\\
&z_9 =\\
&-(1 + y_1 y_2 + y_4 + y_1 y_2 y_4 + y_6 + y_1 y_2 y_6 + y_1 y_8 + y_3 y_8 + 
  y_1 y_2 y_3 y_8 + y_7 y_8 + y_1 y_2 y_7 y_8 + y_4 y_7 y_8 + y_1 y_2 y_4 y_7 y_8 -
   y_6 y_8 y_9)\\
   \hline
&(-1 - y_1 y_2 - y_4 - y_1 y_2 y_4 - y_6 - y_1 y_2 y_6 - 
  y_1 y_8 - y_3 y_8 - y_1 y_2 y_3 y_8 - y_7 y_8 - y_1 y_2 y_7 y_8 - y_4 y_7 y_8 - 
  y_1 y_2 y_4 y_7 y_8\\
& - y_9 - y_1 y_2 y_9 - y_4 y_9 - y_1 y_2 y_4 y_9 + 
  y_4 y_5 y_6 y_9 + y_1 y_2 y_4 y_5 y_6 y_9 - y_1 y_8 y_9 - y_3 y_8 y_9 - 
  y_1 y_2 y_3 y_8 y_9 + y_1 y_5 y_6 y_8 y_9\\
& + y_3 y_5 y_6 y_8 y_9 + 
  y_1 y_2 y_3 y_5 y_6 y_8 y_9 - y_7 y_8 y_9 - y_1 y_2 y_7 y_8 y_9 - 
  y_4 y_7 y_8 y_9 - y_1 y_2 y_4 y_7 y_8 y_9 + y_4 y_5 y_6 y_7 y_8 y_9 + 
  y_1 y_2 y_4 y_5 y_6 y_7 y_8 y_9) 
\end{aligned}
\fontsize{12}{12}
\ee  
and 
\be
\begin{tiny}
\begin{aligned}
  &\delta(y)\\ & \=1 - y_1 - y_3 - y_1 y_2 y_3 - y_4 - y_1 y_5 - y_3 y_5
  - y_1 y_2 y_3 y_5 - y_6 + 
y_1 y_5 y_6 + y_3 y_5 y_6 + y_1 y_2 y_3 y_5 y_6 \\ & \hspace{0.6cm}
- y_7 - y_1 y_2 y_7 + y_4 y_7 + 
 y_1 y_2 y_4 y_7 - y_4 y_5 y_6 y_7 - y_1 y_2 y_4 y_5 y_6 y_7 - y_9 + y_4 y_9 - 
 y_4 y_5 y_6 y_9 \\ &
 \= (z_4)^{-1} - (z_4 z_6)^{-1} - (z_4 z_8)^{-1} + (z_4 z_6 z_8)^{-1}
 + (z_3 z_9)^{-1} - (z_2 z_3 z_9)^{-1} + (z_1 z_2 z_3 z_9)^{-1}
 - (z_4 z_9)^{-1} - (z_3 z_5 z_9)^{-1} \\ & + (z_2 z_3 z_5 z_9)^{-1}
 - (z_1 z_2 z_3 z_5 z_9)^{-1} + (z_4 z_5 z_9)^{-1} + (z_4 z_6 z_9)^{-1}
 - (z_3 z_8 z_9)^{-1} + (z_2 z_3 z_8 z_9)^{-1} - (z_1 z_2 z_3 z_8 z_9)^{-1}
 \\ & + (z_4 z_8 z_9)^{-1} + (z_3 z_5 z_8 z_9)^{-1} - (z_2 z_3 z_5 z_8 z_9)^{-1}
 + (z_1 z_2 z_3 z_5 z_8 z_9)^{-1} - (z_4 z_5 z_8 z_9)^{-1}
 - (z_4 z_6 z_8 z_9)^{-1} + (z_3 z_7 z_8 z_9)^{-1} \\ & 
 - (z_2 z_3 z_7 z_8 z_9)^{-1} + (z_1 z_2 z_3 z_7 z_8 z_9)^{-1}
 - (z_3 z_4 z_7 z_8 z_9)^{-1} + (z_2 z_3 z_4 z_7 z_8 z_9)^{-1}
 - (z_1 z_2 z_3 z_4 z_7 z_8 z_9)^{-1} \\ & - (z_3 z_5 z_7 z_8 z_9)^{-1}  
 + (z_2 z_3 z_5 z_7 z_8 z_9)^{-1} - (z_1 z_2 z_3 z_5 z_7 z_8 z_9)^{-1}
 + (z_3 z_4 z_5 z_7 z_8 z_9)^{-1} - (z_2 z_3 z_4 z_5 z_7 z_8 z_9)^{-1} \\ &
 + (z_1 z_2 z_3 z_4 z_5 z_7 z_8 z_9)^{-1} + (z_3 z_4 z_6 z_7 z_8 z_9)^{-1}
 - (z_2 z_3 z_4 z_6 z_7 z_8 z_9)^{-1} + (z_1 z_2 z_3 z_4 z_6 z_7 z_8 z_9)^{-1} \,.
\end{aligned}
\end{tiny}
\ee
\end{example}

The next lemma relates the effect of the specialization~\eqref{yt} to
$z_i=z_i(y)$. 

\begin{lemma}
\label{lem.zt}
If $A$ is the matrix associated to a braid word, then the Equations~\eqref{eq:nahm}
with $y_i$ specialized as in~\eqref{yt} has solution $z_i=t^{\ve_i}$.
\end{lemma}

\begin{proof}
We follow the notation of Lemma~\ref{lem.yt}. 
With an abuse of notation, denote $j\in\gamma_i$ for
every crossing $j$ encountered by $\gamma$ (as usual counted with multiplicity).
Let $\gamma_i(\text{over})$ and $\gamma_i(\text{under})$ denote the set of over-
and under-crossings encountered by $\gamma$ respectively, and $\sigma_i(j)=\pm1$
to be positive if and only if $j\in\gamma_i(\text{over})$.
We can describe the matrix $A$ in similar terms but this just stores the
over-crossings. Explicitly, we have:
\be
A_{ii}=(1+s(i))/2, \qquad 
A_{ij}=s(i)(1+\sigma_i(j))/2 \qquad (i\neq j) \,.
\ee

We can then compute the RHS of the Equations~\eqref{eq:nahm} when we substitute
$z_i=t^{\ve_i}$ and $y_i$ as in Lemma~\ref{lem.yt} and, since
$d_i=-\sum_{j\in\gamma_i}\ve_j$, we find that
\be
  -(t^{\ve_i}-1)
  \prod_{j\in\gamma_i}
  t^{-s(i)\ve_j/2}
  \prod_{j\in\gamma_i(\text{over})}
  t^{s(i)\ve_j}
  \=
  (1-t^{\ve_i})
  \prod_{j\in\gamma_i}
  t^{\sigma_i(j)s(i)\ve_j/2} \,.
\ee

If we consider the representation of $\pi_1(S^3\smallsetminus K)\to\BC^\times$
with the loop around the knot having monodromy $t^{1/2}$,
then we can compute the
monodromy of the cycle $\gamma$ in two ways.
Either we can push the contour away or pull it towards us.
These two computations
give $\prod_{j\in\gamma_i(\text{over})}t^{\ve_j/2}
=\prod_{j\in\gamma_i(\text{under})}t^{\ve_j/2}$ respectively.
Therefore, we see that the RHS of Equations~\eqref{eq:nahm} under the
specialization is equal to $1-t^{\ve_i}$, which is equal to the LHS of
Equations~\eqref{eq:nahm} under the same specialization.

Therefore, for $y_i=(t^{\ve_i}-1)t^{s(i)d_i/2}$, the specialization for the
quantum invariant, $z_i=t^{\ve_i}$ is the unique solution to
Equations~\eqref{eq:nahm} from Theorem~\ref{thm:lag.inv}.
\end{proof}

We end this section with an example that includes everything about the
$q\to1$ limit of the state-sum, its deformation, and its rationality properties 
in one place. 

\begin{example}
\label{ex.41}
Consider the figure eight knot $4_1$ described by the closure of a braid given by
$\sigma_1^{-1} \sigma_2 \sigma_1^{-1} \sigma_2$
on three strands.
This is pictured as follows:
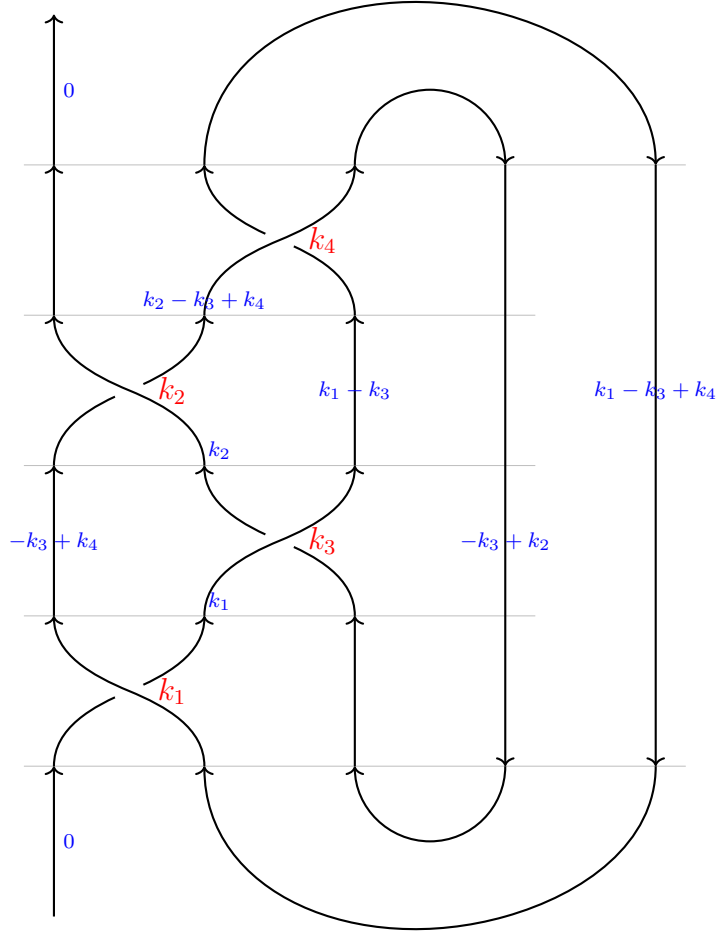
\begin{center}
\begin{tikzpicture}[scale=2]
\foreach \x in {2,3,4} \draw[yshift=\x cm,opacity=0.25](-0.2,0)--(3.2,0);

\draw[yshift=1 cm,opacity=0.25](-0.2,0)--(4.2,0);
\draw[yshift=5 cm,opacity=0.25](-0.2,0)--(4.2,0);
\draw[thick,->,xshift=0cm,yshift=1cm](0,0) to[out angle=90, in angle=-90,
curve through= {(0.5,0.5)}] (1,1);
\filldraw[white,xshift=0cm,yshift=1cm](0.5,0.5) circle (0.1cm) node[red,right]
{$\;\; k_1$};
\draw[thick,->,xshift=0cm,yshift=1cm] (1,0) to[out angle=90, in angle=-90,
curve through = {(0.5,0.5)}] (0,1);

\draw[thick,->,xshift=2cm,yshift=1cm](0,0)--(0,1);

\draw[thick,->,xshift=1cm,yshift=2cm] (1,0) to[out angle=90, in angle=-90,
curve through = {(0.5,0.5)}] (0,1);
\filldraw[white,xshift=1cm,yshift=2cm](0.5,0.5) circle (0.1cm) node[red,right]
{$\;\; k_3$};
\draw[thick,->,xshift=1cm,yshift=2cm](0,0) to[out angle=90, in angle=-90,
curve through= {(0.5,0.5)}] (1,1);

\draw[thick,->,xshift=0cm,yshift=2cm](0,0)--(0,1);

\draw[thick,->,xshift=0cm,yshift=3cm](0,0) to[out angle=90, in angle=-90,
curve through= {(0.5,0.5)}] (1,1);
\filldraw[white,xshift=0cm,yshift=3cm](0.5,0.5) circle (0.1cm) node[red,right]
{$\;\; k_2$};
\draw[thick,->,xshift=0cm,yshift=3cm] (1,0) to[out angle=90, in angle=-90,
curve through = {(0.5,0.5)}] (0,1);

\draw[thick,->,xshift=2cm,yshift=3cm](0,0)--(0,1);

\draw[thick,->,xshift=1cm,yshift=4cm] (1,0) to[out angle=90, in angle=-90,
curve through = {(0.5,0.5)}] (0,1);
\filldraw[white,xshift=1cm,yshift=4cm](0.5,0.5) circle (0.1cm) node[red,right]
{$\;\; k_4$};
\draw[thick,->,xshift=1cm,yshift=4cm](0,0) to[out angle=90, in angle=-90,
curve through= {(0.5,0.5)}] (1,1);

\draw[thick,->,xshift=0cm,yshift=4cm](0,0)--(0,1);

\draw[thick,->](2,5) to[out angle=90, in angle=90, curve through= {(2.5,5.5)}] (3,5);
\draw[thick,->](3,5)--(3,1);
\draw[thick,->](3,1) to[out angle=-90, in angle=-90, curve through= {(2.5,0.5)}] (2,1);

\draw[thick,->](1,5) to[out angle=90, in angle=90, curve through= {(3,6)}] (4,5);
\draw[thick,->](4,5)--(4,1);
\draw[thick,->](4,1) to[out angle=-90, in angle=-90, curve through= {(3,0)}] (1,1);

\draw[thick,->](0,0)--(0,1);
\draw[thick,->](0,5)--(0,6);
\draw[blue](0.1,0.5) node {\tiny $0$};
\draw[blue](1.1,2.1) node {\tiny $k_1$};
\draw[blue](2,3.5) node {\tiny $k_1-k_3$};
\draw[blue](4,3.5) node {\tiny $k_1-k_3+k_4$};
\draw[blue](0,2.5) node {\tiny $-k_3+k_4$};
\draw[blue](1,4.1) node {\tiny $k_2-k_3+k_4$};
\draw[blue](3,2.5) node {\tiny $-k_3+k_2$};
\draw[blue](1.1,3.1) node {\tiny $k_2$};
\draw[blue](0.1,5.5) node {\tiny $0$};
\end{tikzpicture}
\end{center}

This leads to the quantum invariant
\be
\begin{tiny}
\begin{aligned}
  &\sum_{k\in\BZ_{\geq0}^4}
  (-1)^{k_3+k_4}
  q^{2k_1k_2-k_1k_3-k_2k_3-k_3k_4+3k_3(k_3+1)/2+k_4(k_4-1)/2-k_1-k_2-k_4}
  t^{-k_1-k_2+1}\\
  &\qquad\bigg[\!\!\begin{array}{c} k_1-k_3+k_4\\k_1\end{array}\!\!\bigg]_q
  \bigg[\!\!\begin{array}{c} k_1\\k_3\end{array}\!\!\bigg]_q
  \bigg[\!\!\begin{array}{c} k_2-k_3+k_4\\k_4\end{array}
  \!\!\bigg]_q(t;q^{-1})_{k_1}(q^{k_3-k_4}t;q^{-1})_{k_2}
  (q^{k_3-k_2}t;q^{-1})_{k_3}(q^{-k_1+k_3}t;q^{-1})_{k_4}\,.
\end{aligned}
\end{tiny}
\ee


This braid word leads to the deformed sum
\be
\label{delta41}
\begin{aligned}
  \frac{1}{\delta}
  &\=
  \sum_{k\in\BZ_{\geq0}^4}
  \binom{k_1-k_3+k_4}{k_1}
  \binom{k_2}{k_2}
  \binom{k_1}{k_3}
  \binom{k_2-k_3+k_4}{k_4}
  y^k\\
  &\=
  \frac{1}{1-y_1-y_2-y_4+y_1y_2-y_1y_2y_3y_4}\,.
\end{aligned}
\ee
Specializing
\be
\label{y41t}
y_1\=y_2\=t^{-1}-1, \qquad y_3\=y_4\=t-1
\ee
gives 
\be
  \frac{t^{-1}}{-t^{-1}+3-t} 
\ee
in agreement with the fact that the Alexander polynomial of $4_1$ is $-t^{-1}+3-t$.
The combinatorial data associated with this braid word is
\be
A
\=
\begin{pmatrix}
  1 & 0 & 1 & 0\\
  0 & 1 & 0 & 1\\
  -1 & 0 & 0 & -1\\
  1 & 0 & 0 & 1
\end{pmatrix}\!,
\quad
B
\= \begin{pmatrix}
  0 & 0 & 0 & 1\\
  0 & 0 & 1 & 0\\
  0 & -1 & -1 & -1\\
  0 & 1 & 0 & 0
\end{pmatrix}\!,
\quad
\ve \= \begin{pmatrix}
-1 \\ -1 \\ 1 \\ 1  
\end{pmatrix}\!.
\ee
and
\be
\begin{aligned}
  s&\=(\phantom{-}1,\phantom{-}1,-1,\phantom{-}1)\,,\\
  d&
    \=(\phantom{-}0,\phantom{-}0,\phantom{-}0,\phantom{-}0)\,.
\end{aligned}
\ee
The Nahm equations associated to this matrix are 
\be
\label{nahm41}
1-z_1\=-y_1z_1z_3\,,\quad
  1-z_2\=-y_2z_2z_4\,,\quad
  1-z_3\=-y_3z_1^{-1}z_4^{-1}\,,\quad
  1-z_4\=-y_4z_1z_4\,.
\ee
These equations have a unique solution
\be
\begin{array}{llll}
  z_1&\=\frac{-1 - y_1 y_3}{-1 + y_1 - y_1 y_3 y_4}\,,\quad
  &z_2\!&\=\frac{-1 + y_1 + y_4}{-1 + 
  y_1 + y_2 - y_1 y_2 + y_4 + y_1 y_2 y_3 y_4}\,,\\
  \\
  z_3&\=\frac{1 + y_3 - y_3 y_4}{1 + y_1 y_3}\,,\quad
  &z_4\!&\=\frac{1 - y_1 + y_1 y_3 y_4}{1 - y_1 - y_4}\,.
\end{array}
\ee
This solution can be used to express $\delta$ as follows
\be
\begin{aligned}
\delta &\= 1-y_1-y_2-y_4+y_1y_2-y_1y_2y_3y_4\\ &\=
z_2^{-1} - (z_1 z_2)^{-1} - (z_2 z_3)^{-1} + (z_1 z_2 z_3)^{-1} + (z_1 z_2 z_4)^{-1}
\,.
\end{aligned}
\ee
Specializing $z_1^{-1}=z_2^{-1}=z_3=z_4=t$ in Equations~\eqref{nahm41}
gives~\eqref{y41t}.
\end{example}


\section{A review of the Habiro ring of an \'etale $\BZ[t]$-algebra}
\label{sec.hab}

\subsection{Basics on $p$-adic completions}

If $R$ is a ring of characteristic $0$, then for each prime $p\in\BZ$, we can
construct two rings: the reduction of $R$ modulo $p$ defined by $R/pR$, and the
completion of $R$ at $p$ defined by $R_p=\varprojlim_{n} R/p^nR$.
From its definition, we have natural maps $R\to R_p\to R/pR$.
An immediate consequence of the binomial theorem implies that if a ring $R$ has characteristic
$0$, then the $p$-power map $F_p:R/pR\to R/pR$ defined so that $x\mapsto x^p$ is a
ring homomorphism.
A \emph{Frobenius lift} is a homomorphism $\varphi_p:R_p\to R_p$ such that
$\varphi_p(x)\equiv x^p\pmod{p}$.

In this paper, we are interested in $R=\BZ[t^{\pm1},1/\Delta(t)]$ where $\Delta(t)\in\BZ[t^{\pm1}]$.
In this example, we see that elements in $Q(t)\in R$ can be expressed
$Q(t)=P(t)/\Delta(t)^n$ for some $n\in\BZ_{\geq0}$ and $P(t)\in\BZ[t^{\pm1}]$.
We see that $Q(t)\in pR$ when $P(t)\in p\BZ[t^{\pm1}]$.
Then we can express elements in $R_p$ non-uniquely by sums
\be
  \sum_{k=0}^{\infty}
  Q_{k}(t)p^k
  \in
  R_{p}\,,
\ee
where $Q_k(t)\in R$. We see that, if we choose some\footnote{One can simply
  think of any of the subsets $\BZ\subseteq\BQ\subseteq\BQ_p\subseteq\BC_p$.}
$t\in\BC_p$ with $|t|_p\leq 1$ and $|\Delta(t)|_p\geq 1$, elements in $R_p$ can
be evaluated at $t$.

There are many different Frobenius lifts $\varphi_p:R_p\to R_p$.
There is a somewhat canonical lift given by $a=0$ in the following:
\begin{proposition}
  If $R=\BZ[t^{\pm1},1/\Delta(t)]$ for $\Delta(t)\in\BZ[t^{\pm1}]$, then for $a\in\BZ[t^{\pm1}]$ the
  map $t\mapsto (t+a)^p-a^p$ defines a Frobenius lift on $R_p$.
\end{proposition}
\begin{proof}
  Firstly, let $\delta(t)=p^{-1}(\Delta(t)^p-\Delta((t+a)^p-a^p))\in\BZ[t^{\pm1}]$. Then
  we find that
\be
  \frac{1}{\Delta((t+a)^p-a^p)}
  \=
  \frac{1}{\Delta(t)^p-p\delta(t)}
  \=
  \sum_{k=0}^{\infty}\frac{\delta(t)^k}{\Delta(t)^{p(k+1)}}p^k\in R_p\,.
\ee
Therefore, this map sends $1/\Delta(t)$ to $R_p$ and of course sends $\BZ[t^{\pm1}]$
to $\BZ[t^{\pm1}]$.
This can be used to define a map $R_p\to R_p$
\end{proof}
\begin{remark}
Note that $R_p$ is not isomorphic to $\BZ_p[t^{\pm1},1/\Delta(t)]$.
For example, $\BZ_p[t]\subseteq\BZ[t]_p\subseteq\BZ_p[\![t]\!]$ and we have
\be
\begin{aligned}
  \sum_{\ell=0}^{L}
  \sum_{k=\ell}^\infty p^kt^\ell
  \=
  \sum_{k=0}^{L} p^k\frac{1-t^{k+1}}{1-t}
  +
  \sum_{k=L+1}^\infty p^k\frac{1-t^{L+1}}{1-t}
  &\in\BZ_p[t]\,,\\
  \sum_{\ell=0}^\infty
  \sum_{k=\ell}^\infty p^kt^\ell
  \=\sum_{k=0}^{\infty} p^k\frac{1-t^{k+1}}{1-t}
  \=\frac{1}{(1-p)(1-pt)}
  &\in\BZ[t]_p\,,\\
  \frac{1}{(1-p)(1-t)}
  \=\sum_{\ell,k=0}^\infty
  p^k t^\ell
  &\in\BZ_p[\![t]\!]\,.\\
\end{aligned}
\ee
We see that while the degree of $t$ is unbounded for the element in $\BZ[t]_p$,
for a fixed power of $p$ the degree in $t$ has a finite bound.
Consequently, elements in $\BZ[t]_p$ can be evaluated at any $|t|_p\leq 1$,
whereas elements of $\BZ_p[\![t]\!]$ can only be evaluated at $|t|_p<1$.
\end{remark}

\begin{proposition}
The ring $\BZ[t^{1/p^\infty},\tfrac{1}{\Delta(t)}]_p$ has a Frobenius lift
$t\mapsto t^p$ and this is an automorphism.
\end{proposition}

\begin{proof}
By induction, we see that there exists $P_k(t)\in\BZ[t^{1/p^\infty},1/\Delta(t)]$
such that
\be
  \frac{1}{\Delta(t^{1/p})}
  \=
  \sum_{k=0}^n P_k(t)p^k\pmod{p^{n+1}}\,,
\ee
since
\be
  \frac{1-\Delta(t^{1/p})\sum_{k=0}^n P_k(t)p^k}{p^{n+1}\Delta(t^{1/p})}
  \equiv
  \frac{1-\Delta(t^{1/p})\sum_{k=0}^n P_k(t)p^k}{p^{n+1}\Delta(t)}
  \Delta(t^{1/p})^{p-1}\pmod{p}\,.
\ee
\end{proof}

\subsection{$\BZ[t]$-\'etale algebras, their Habiro rings, and elements}

In this section, we will recall the basic properties of Habiro rings of affine space
and étale maps over it.

Firstly, we recall the relation between the Habiro ring and the original definition.
This relation is governed by the following.

\begin{proposition}
If $\varphi_p:R\to R$ and is an isomorphism onto its
image\footnote{The main condition is that the codomain does not need a
  $p$-adic completion.} for all primes $p$, then
\be
  \calH_{R/\BZ[t^{\pm 1}]}
  \cong
  \varprojlim_n R[q]/(q;q)_nR[q]\,.
\ee
\end{proposition}
\begin{proof}
This follows from the isomorphism
\be
\begin{small}
\begin{aligned}
  \varprojlim_n R[q]/(q;q)_nR[q]
  \cong\Big\{&(f_m)_m\;\in\;\prod_{m\geq 1}R[\z_m][\![q-\z_m]\!]
  \;\Big|\;\text{for all }m\in\BZ_{>0}\text{ and }p\in\BZ_{>0}\text{ prime}\\
  &f_m(q-\z_{pm}+\z_{pm}-\z_m)\=f_{pm}(q-\z_{pm})
  \in R_p[\z_{pm}][\![q-\z_{pm}]\!]\Big\}\,,
\end{aligned}
\end{small}
\ee
combined with the isomorphism $(f_m)_m\to (\varphi_m f_m)_m$, where
$\varphi_m=\prod_{p}\varphi_p^{v_p(m)}$.
\end{proof}

This shows that the Habiro ring of $\BZ[t]$, with Frobenius lifts $t\mapsto t^p$,
is described by the isomorphism
\be
  \varphi:\calH_{\BZ[t]}\to\varprojlim_{n}\BZ[t,q]/(q;q)_n\BZ[t,q]\,.
\ee
Therefore, using this isomorphism, we can easily construct all elements of
$\calH_{\BZ[t]}$ by taking
\be
  \varphi^{-1}\Big(
  \sum_{n=0}^\infty a_n(t;q)(q;q)_n\Big)\,,
\ee
for any sequence $a_{n}(t;q)\in\BZ[t,q]$ (noting of course that a choice of
$a_{n}(t;q)$ is not unique).

These rings satisfy similar properties to Habiro's original ring $\calH_\BZ$.
This is summarized in the following:

\begin{proposition}
\label{prop.hab11}
The canonical maps
\be
  \calH_{R/\BZ[t^{\pm 1}]}\to R[\![q-1]\!]\,,\qquad
  \calH_{R/\BZ[t^{\pm 1}]}\to \prod_{m=1}^{\infty}R[t^{1/m},\z_m]
\ee
are injections.
\end{proposition}
That is, an element of the ring $\calH_{R/\BZ[t^{\pm 1}]}$ is uniquely determined
by either its expansion at $q$ near 1, or by its value at roots of unity.

\begin{proof}
The first injection follows from Equation~\eqref{gluefmp}.
The second follows from the first.
Indeed, if an element $f$ satisfies $f(\z_m)=0$ for all $m$, and
$f_1(t,q-1)=\sum_{k=0}^\infty c_k (q-1)^k$, we by assumption have that $c_0=0$.
Then by Equation~\eqref{gluefmp} we find that $c_1$ contains infinitely many prime factors and hence
$c_1=0$, then $c_k=0$ follows by the same reasoning and induction.
\end{proof}

We now recall a basic construction of elements in Habiro rings
from~\cite[Sec.1.5]{GW:explicit}. The main idea is to replace any factorials in the
expansions of a rational function by $q$-factorials.
This leads to elements in the Habiro ring
defined over the same ring of rational functions.
As a fundamental example, consider the rational function
\be
\label{fcalA}
f_{\calA,Q}(t)
\= \frac{1}{1- \sum_{\a \in \calA} (-1)^{Q_{\a\a}}t_\a}
\= \sum_{k \in \BZ_{\geq 0}^\calA}(-1)^{\diag(Q)k}
\frac{(\sum_\a k_\a)!}{\prod_\a (k_\a)!} \prod_\a t_\a^{k_\a}
\ee
in a finite set $\calA$ of variables $t_\a$, where $t=(t_\a)_{\a \in \calA}$,
$k=(k_\a)_{\a \in \calA}$ and $Q$ is an integral symetric matrix (which is
currently all redundant expect for the diagonal modulo two).
Replacing the
factorials with $q$-factorials and adding the quadratic from $Q$, we obtain
a $q$-deformation defined by
\be
\label{fcalAq}
f_{\calA,Q}(t,q) \=  \sum_{k \in \BZ_{\geq 0}^\calA}
(-1)^{\diag(Q)k}q^{\frac{1}{2}k^{t}Qk+\frac{1}{2}\diag(Q)k}
\frac{(q;q)_{\sum_\a k_a}}{\prod_\a (q;q)_{k_\a}} \prod_\a t_\a^{k_\a} \,.
\ee
\begin{theorem}
\label{thm.ct2}\cite{GW:explicit}
For every finite set $\calA$ and $Q$ integral symmetric matrix indexed by $\calA$,
we have $f_{\calA,Q}(t,q) \in \calH_{\BZ[t,f_{\calA,Q}(t)]/\BZ[t]}$.
\end{theorem}

This basic example can be used to construct many new elements by specializing
various $t_\a=1$ or taking push-forwards.

\begin{example}
\label{ex.fun}
An example of the above theorem is $f(t,q)=(1-t)^{-1}\in\calH_\BZ[t,(1-t^{-1})]$,
since
\be
  f_m
  \=
  \frac{1}{1-t^{1/m}}
  \=
  \frac{1}{1-t}(1+t^{1/m}+t^{2/m}+\cdots+t^{(m-1)/m})
  \in
  \mathbb{Z}[t^{1/m},(1-t)^{-1}]\,.
\ee
\end{example}

Another class of examples, comes from the binomial sums originating from the
application of Lagrange's inversion theorem in Section~\ref{sub.lag}.
Indeed, letting $R$ be the algebra $\BZ[z^{\pm1},\delta^{-1},y^{\pm1}]$ modulo the
Equations~\eqref{eq:nahm}, we have the following:

\begin{theorem}
\label{thm.ct3}\cite{GW:explicit}
For any integer matrix $A$, $Q$ an integral quadratic form with diagonal
entries odd, $\nu$ is an integral linear form with odd entries, and any integral
$s_i$, the $q$-hypergeometric sum
\be
  \sum_{k\in\BZ_{\geq0}^N}
  (-1)^{k_1+\cdots+k_N}
  q^{\frac{1}{2}Q(k)+\frac{1}{2}\nu(k)}
  \prod_{i=1}^N
  \bigg[\!\!\begin{array}{c} s_i+(A^{t}k)_i\\k_i\end{array}\!\!\bigg]_{\!q}
  y_i^{k_i}\,,
\ee
defines an element of the Habiro ring $R/\BZ[y]$ with respect to the Frobenius
lifts $y\mapsto y^p$.
\end{theorem}

\begin{remark}
This example will be fundamentally different from $f_{A,B,\nu,\ve}$.
The main difference being the Frobenius action on the variable $t$.
This highlights the interaction between the $q$-Pochhammer symbols and Frobenius lifts.
\end{remark}

\subsection{Cyclotomic expansions}
\label{sub.syclo}

With the basics on Habiro rings and their elements recalled, we can now prove
Lemma~\ref{lem.unique}. First we need the following elementary.

\begin{proposition}
\label{prop:eval}
If $\Delta(1)=1$ and $R=\BZ[t^{\pm1},\Delta(t)^{-1}]$, then
$\mathrm{ev}_{t=1}:\calH_{R/\BZ[t^{\pm1}]}\to\calH_{\BZ}$.
\end{proposition}

\begin{proof}
This follows from the fact that the Frobenius $t\mapsto t^p$ commutes with
$\mathrm{ev}_{t=1}$.
\end{proof}

With this in hand we can complete the proof of the lemma. 
\begin{proof}[Proof of Lemma~\ref{lem.unique}]
Consider the operator
\be
\label{Top}
(Tf)(t,q) \= \frac{f(t,q)-f(1,q)}{t-1} \,.
\ee
Note that if $f_m(t,q-\z_m) \in \BZ[t^{\pm 1/m}, \z_m,
\tfrac{1}{\Delta(t)}][\![q-\z_m]\!]$, then $(f_m(t,q-\z_m)-f_m(1,q-\z_m))/(t-1)
\in \BZ[t^{\pm 1/m}, \z_m,
\tfrac{1}{\Delta(t)}][\![q-\z_m]\!]$. 

From Proposition~\ref{prop:eval} and Example~\ref{ex.fun}, it is easy to see that
if $f \in \calH_{\BZ[t^{\pm 1}, \Delta(t)^{-1}]/\BZ[t^{\pm 1}]}$,
then $Tf\in \calH_{\BZ[t^{\pm 1}, \Delta(t)^{-1},(t-1)^{-1}]/\BZ[t^{\pm 1}]}$.
Moreover, since $(Tf)_m\in\BZ[t^{\pm 1/m}, \z_m,
\tfrac{1}{\Delta(t)}][\![q-\z_m]\!]$, we see that in fact
$Tf \in \calH_{\BZ[t^{\pm 1}, \Delta(t)^{-1}]/\BZ[t^{\pm 1}]}$.

Hence, we have
\be
\label{ff1}
f(t,q) = a_0(q) + (t^{-1}-1) f_1(t,q)
\ee
with $a_0(q) = f(t,1) \in \calH_\BZ$ and $f_1(t,q) = -t (Tf)(t,q) \in 
\calH_{\BZ[t^{\pm 1}, \Delta(t)^{-1}]/\BZ[t^{\pm 1}]}$. It is easy to see that $f$
uniquely determines $a_0$ and $f_1$. Shifting $f_1(t,q)$ to $f_1(qt,q)$ and
iterating~\eqref{ff1}, it follows that
\be
\label{ff2}
f(t,q) = a_0(q) + (t^{-1}-1) (a_1(q) + (q t^{-1}-1) f_2(t,q)) 
\ee
for unique $a_0(q), a_1(q) \in \calH_\BZ$ and $f_2(t,q)\in
\calH_{\BZ[t^{\pm 1}, \Delta(t)^{-1}]/\BZ[t^{\pm 1}]}$. Continuing by induction, we
obtain an expansion of the form
\be
\label{ff3}
f(t,q) = \sum_{k=0}^\infty (t^{-1};q)_k a_k(q)  
\ee
with $a_k(q)$ uniquely determined by $f(t,q)$. This gives a well-defined injective
ring homomorphism
\be
\label{finj}
\calH_{\BZ[t^{\pm 1}, \Delta(t)^{-1}]/\BZ[t^{\pm 1}]} \hookrightarrow
\varprojlim_k \;\calH_\BZ[t^{\pm1}]/((t^{-1};q)_k) 
\ee
and concludes the first part of the lemma.

For the second part, observe that
\be
f(q^n,q) \= \sum_{k=0}^n (q^{-n};q)_k \, a_k(q), \qquad (n \geq 0) \,.
\ee
Hence if $f(q^n,q) \in \BZ[q^{\pm 1}]$, then it follows by induction on $n$
that $a_n(q) \in \tfrac{1}{(q;q)_n} \BZ[q^{\pm 1}] \cap \calH_\BZ$. But
\be
\tfrac{1}{(q;q)_n} \BZ[q^{\pm 1}] \cap \calH_\BZ \= \BZ[q^{\pm 1}] \,.
\ee
This concludes the proof of the lemma. 
\end{proof}


\section{Proofs}
\label{sec.proofs}

\subsection{Deformation of Theorem~\ref{thm.1}}

Recall the state-sum $f_{A,B,\nu,\ve}(t,q)$ from~\eqref{ABt} of the colored Jones
polynomial of a knot given by the closure of a braid $\beta$ in $r$ strands with $N$
crossings. We will consider
multivariable versions of these state-sums that, in addition to $t$ and $q$,
also depend on $x=(x_1,\dots,x_N)$:
\be
\label{ABxt}
\begin{aligned}
&f_{A,B,\nu,\ve}(x,t,q)\\
&=
t^{\frac{r-1}{2}}\!\!
\sum_{k\in\BZ_{\geq0}^N}
q^{\frac{1}{2}Q(k)-\nu(k)}
\prod_{i=1}^N(-\ve_iq^{\frac{-\ve_i-1}{4}}t^{-\frac{1}{2}})^{k_i}
\bigg[\!\!\begin{array}{c} (A^{t}k)_i \\
k_i\end{array}\!\!\bigg]_{q}\!\!
(q^{-(B^{t} k)_i}t;q^{-1})_{k_i}
t^{-\frac{\ve_i}{2} ((A^{t}k)_i + (B^{t}k)_i+1)} x_i^{k_i}.
\end{aligned}
\ee

Note that
\be
f_{A,B,\nu,\ve}(x,t,q) \in \BZ[t^{\pm 1}, q^{\pm 1}][\![x]\!], \qquad
f_{A,B,\nu,\ve}(x,t,1) \in \BZ[t^{\pm 1}][\![x]\!]
\ee
satisfy $f_{A,B,\nu,\ve}(0,t,q)=1$ and consequently $f_{A,B,\nu,\ve}(0,1,1)=1$.
Moreover, $f_{A,B,\nu,\ve}(1,t,q)=f_{A,B,\nu,\ve}(t,q)$.
Consider the ring
\be
R \= \BZ[z^{\pm1},t^{\pm1},x^{\pm1},\d^{-1}]/(1-z_1+x_1y_1
\prod_jz_j^{A_{1j}},\cdots,1-z_N+x_Ny_N\prod_jz_j^{A_{Nj}})\,,
\ee
with $\d$ as in Equation~\eqref{deltay} and
$y_i=(t^{\ve_i}-1)t^{s(i)d_i/2}$ (as in Lemma~\ref{lem.yt}). The ring $R_p$ has
Frobenius lifts $\varphi_p$ raising $x,t$ to their $p$-th powers.
For an integer $b$ and positive integer $m$, we define
\be
\label{bm}
\langle b\rangle_m\equiv b\pmod{m}, \qquad 0\leq\langle b\rangle_m<m \,.
\ee

\begin{theorem}
\label{thm.1.def}
The formula $f_{A,B,\nu,\ve}(x,t,q)$ defines an element of the Habiro ring $\calH_R$.
Moreover,
\be
\label{trhasse}
\begin{aligned}
\d f_{A,B,\nu,\ve}(x^{\frac{1}{m}},t^{\frac{1}{m}},\z_m)
&\;=\!\!\!\!\!\!
\sum_{n\in(\BZ/m\BZ)^N}
\z_m^{\frac{1}{2}Q(n)-\nu(n)}
\prod_{i=1}^N(-\ve_i\z_m^{\frac{-\ve_i-1}{4}}t^{-\frac{1}{2m}})^{n_i}
\bigg[\!\!\begin{array}{c} \langle(A^{t}n)_i\rangle_m \\
n_i\end{array}\!\!\bigg]_{\z_m}
z_i^{\lfloor\frac{(A^{t}n)_i-n_i}{m}\rfloor}\\
&\qquad\qquad\times(\z_m^{-(B^{t} n)_i}t^{\frac{1}{m}};\z_m^{-1})_{n_i}
t^{-\frac{\ve_i}{2m} ((A^{t}n)_i + (B^{t}n)_i+1)+\frac{r-1}{2m}}
x_i^{\frac{n_i}{m}}\!\!.
\end{aligned}
\ee
\end{theorem}

Over the next sections, we will prove this theorem. However, before that, we will
give a lemma needed for the main theorem~\ref{thm.1}.

\begin{lemma}
\label{lem:def2A}
If $A,B,\ve$ come from the braid presentation of a knot $K$ (as described in
Section~\ref{sub.state}), then the evaluation map $x=1$ induces a ring map
\be
  \mathrm{ev}_{x=1}:R\to\BZ[t^{\pm 1}, \Delta_K(t)^{-1}]\,.
\ee
\end{lemma}

\begin{proof}
Firstly, from Lemma~\ref{lem.zt}, we see that we obtain a map from $R$ to
$\BZ[t^{\pm 1}, \d^{-1}]$.
Then, since $\d^{-1}=f_{A,B,\nu,\ve}(0,t,1)$, it follows from Equation~\eqref{fAB1}
that $\d=\Delta_K(t)$.
\end{proof}

\subsection{Basics on the expansion of the $q$-Pochhammer symbol near
  roots of unity}
\label{sub.qp}

In this section, we review some basic properties of $q$-Pochhammer symbols and
their expansions near roots of unity. To begin, the elementary identity
\be
\label{qpe1}
(t;q)_k
\=
\frac{(t;q)_\infty}{(q^k t;q)_\infty} \qquad (k \in \BZ_{\geq 0})
\ee
allows us to deduce expansions of a finite $q$-Pochhammer symbol from the
corresponding ones of the infinite $q$-Pochhammer symbol.

This is complemented by the elementary but useful congruence product identity, which determines the behavior near $q=\z_m$ by that at $q=1$,
\be
\label{qpe2}
(t;q)_\infty \= \prod_{j=0}^{m-1} (q^jt;q^m)_\infty \qquad (m \in \BZ_{>0}) \,.
\ee

Next, we recall the asymptotic expansions near $1$. This is standard, and a
detailed proof can be found e.g., in~\cite[Thm.2.2]{Ban}.

\begin{lemma}
\label{lem.asy}
For $a \in \BZ$ and $q=e^\hbar$ with $\hbar \to 0$, we have:
\be
\label{qaxinf}
\begin{aligned}
(q^at;q)_\infty
& \sim
\exp\Big(
\sum_{\ell=0}^\infty
\frac{B_\ell(a)}{\ell!}\Li_{2-\ell}(t)\hbar^{\ell-1}
\Big) \,.
\end{aligned}
\ee
\end{lemma}

To understand the expansion of $f_{A,B,\nu,\ve}(x,t,q)$ for $q$ near roots of unity,
we use the following:

\begin{lemma}
\label{lem:pochhammers}
For fixed $m\in\BZ_{>0}$ and $n\in\{0,1,\dots,m-1\}$, $k\in\BZ_{\geq0}$, and
$a,b\in\BZ$, there exists $D_{m,n,b}(\z_m,k,a,t,\partial_t,q-\z_m)\in
\BQ[\z_m,k,a,t^{\frac{1}{m}},\partial_t][\![q-\z_m]\!]$ such that
\be
(q^{am+b}t^{\frac{1}{m}};q)_{km+n}
\=
D_{m,n,b}(\z_m,k,a,t,\partial_t,q-\z_m)
(1-t)^k
\ee
where the coefficient of $(q-\z_m)^\ell$ in $D_{m,n,b}$ has $\partial_t$ degree
at most $\ell$ and $t$-degree strictly less than $1$ and
\be
  D_{m,n,b}(\z_m,k,a,t,\partial_t,0)
  \=
  (\z_m^{b}t^{\frac{1}{m}};\z_m)_n\,.
\ee
\end{lemma}

\begin{proof}
Equation~\eqref{qpe1} and~\eqref{qaxinf} imply that
\be\label{eq:Ppol}
\begin{aligned}
(q^at;q)_k
&\= \frac{(q^at;q)_\infty}{(q^{k+a}t;q)_\infty}
\\
&\=(1-t)^k
\exp\Big(
\sum_{\ell=2}^\infty
(B_\ell(a)-B_\ell(k+a))\frac{\Li_{2-\ell}(t)}{\ell!}\hbar^{\ell-1}
\Big)\\
&\=
(1-t)^k
\sum_{\ell=0}^\infty
\sum_{j=0}^{\ell}
P_{\ell,j}(a,k)
(1-t)^{-j}
(q-1)^{\ell}\,,
\end{aligned}
\ee
for some polynomials $P_{\ell,j}(a,k)\in\BQ[a,k]$, since the $\ell=0$ term
in~\eqref{qaxinf} cancels and the $\ell=1$ term with
$B_1(a)=-\frac{1}{2}+a$ and $\Li_1(t)=-\log(1-t)$ gives the $(1-t)^k$ prefactor,
$(1-t)^{\ell-1} \Li_{2-\ell}(t) \in \BZ[t]$ for $\ell \geq 2$,
and the $t$-degree of $(q^at;q)_k$ is $k$. 
Given that $(q^at;q)_k\in\BZ[t,q]$ for all $k$, we see that $P_{\ell,j}(a,k)=0$
for $k\in\BZ_{\geq0}$ with $j>k$.
Therefore, we see that
$Q_{\ell,j}(a,k)=P_{\ell,j}(a,k)/k(k-1)\cdots(k-j+1)\in\BQ[a,k]$.
This implies that
\be
(q^at;q)_k
\=
\sum_{\ell=0}^\infty
\sum_{j=0}^{\ell}
Q_{\ell,j}(a,k)
\partial_t^j
(1-t)^k
(q-1)^{\ell}
\ee
with $Q_{\ell,j}(a,k) \in \BQ[a,k]$. 
Therefore,
\be
D_{1,0,0}(1,k,a,t,\partial_t,q-1)
\=\sum_{\ell=0}^\infty\sum_{j=0}^{\ell}Q_{\ell,j}(a,k)\partial_t^j(q-1)^{\ell}
\ee
gives the proof of the lemma when $m=1$.

For the case of general $m>0$, we use Equations~\eqref{qpe2} and~\eqref{eq:Ppol} to
find that there is some $P_{\ell,j,m,b,n}(a,k,t)\in\BQ[\z_m,a,k,t^{\frac{1}{m}}]$
with $t$-degree strictly less than $1$ such that
\be
\begin{aligned}
&(q^{am+b}t^{\frac{1}{m}};q)_{km+n}
\= \prod_{j=0}^{m-1} (q^{am+b+j}t^{\frac{1}{m}};q^m)_{k+\lceil
\frac{n-j}{m}\rceil} \\
&\=
\prod_{j=0}^{m-1}
(1-\z_m^{b+j}t^{\frac{1}{m}})^k
\sum_{\ell=0}^\infty
\sum_{j=0}^{\ell}
P_{\ell,j}(a+\tfrac{b+j}{m}
,k+\lceil
\tfrac{n-j}{m}\rceil)
(1-\z_m^{b+j}t^{\frac{1}{m}})^{-j}
(q^m-1)^{\ell}\\
&\=
\sum_{\ell=0}^\infty
\sum_{j=0}^{\ell}
P_{\ell,j,m,b,n}(a,k,t)
(1-t)^{k-j}
(q-\z_m)^{\ell}\,.
\end{aligned}
\ee
Since, for all $k$ we find that $(q^{am+b}t^{\frac{1}{m}};q)_{km+n}$ has no poles
at $t=1$ and $P_{\ell,j,m,b,n}(a,k,t)$ has $t$-degree less than $1$, we again
see that $P_{\ell,j,m,b,n}(a,k,t)$ must be divisible by $k(k-1)\cdots (k-j+1)$.
\end{proof}

We also need a version of the above lemma when $t=1$.

\begin{lemma}
\label{lem.poch.t1}
For fixed $m\in\BZ_{>0}$ and $n\in\{0,1,\dots,m-1\}$, $k\in\BZ_{\geq0}$, and
$a,b\in\BZ$, 
there exist $D_{m,n,b}(\z_m,k,a,q-\z_m)\in\BQ[\z_m,k,a][\![q-\z_m]\!]$
such that
\be
(q^{am+b};q)_{km+n}
\=
(-1)^km^{2k}(q-\z_m)^k
\frac{\Gamma(a+k+\lceil\tfrac{b+n}{m}\rceil)}{\Gamma(a+\lceil\tfrac{b}{m}\rceil)}
D_{m,n,b}(\z_m,k,a,q-\z_m)\,,
\ee
and
\be
D_{m,n,b}(\z_m,k,a,0)
\=
(\z_m^{b};\z_m)_n\,.
\ee
\end{lemma}
\begin{proof}
Suppose that $q=\z_me^\hbar$.
Then, from Equations~\eqref{qpe2} and~\eqref{qaxinf} and~\cite[Thm.3.2]{Ban}
we find that
\be
\label{eq:qpinfm}
\begin{aligned}
  &(q^b;q)_\infty
  \;\sim\;
  \sqrt{2\pi}\frac{(-m\hbar)^{\frac{1}{2}-\lceil
      \frac{b}{m}\rceil}}{\Gamma(\lceil\tfrac{b}{m}\rceil)}
  \prod_{\langle-b\rangle_m\neq j=0}^{m-1}
  (1-\z_m^{b+j})^{\frac{1}{2}-\frac{b+j}{m}}
  \exp\Big(\frac{\pi^2}{6m\hbar}\\
  &-\sum_{\ell=2}^\infty
  B_{\ell}(\lceil\tfrac{b}{m}\rceil)
  \frac{B_{\ell-1}(1)}{\ell-1}\frac{(m\hbar)^{\ell-1}}{\ell!}+\!\!\!\!\!\!
  \sum_{\langle-b\rangle_m\neq j=0}^{m-1}
  \sum_{\ell=2}^{\infty}\frac{B_{\ell}(\frac{b+j}{m})}{\ell!}
  \Li_{2-\ell}(\z_m^{b+j})(m\hbar)^{\ell-1}\Big) \,.
\end{aligned}
\ee
Hence, we have
\be
\begin{aligned}
  (q^{am+b};q)_{km+n}
  &\=
  \frac{(q^{am+b};q)_\infty}{(q^{(a+k)m+b+n};q)_{\infty}}\\
  &\;\in\;
  m^{2k+\lceil\frac{b+n}{m}\rceil-\lceil\frac{b}{m}\rceil}
  (-\hbar)^{k+\lceil\tfrac{b+n}{m}\rceil-\lceil\frac{b}{m}\rceil}
  \frac{\Gamma(a+k+\lceil\frac{b+n}{m}\rceil)}{\Gamma(a+\lceil\tfrac{b}{m}\rceil)}\\
  &\qquad\qquad\times
  \bigg(
  \prod_{\langle-b\rangle_m\neq j=0}^{n-1}
  (1-\z_m^{b+j})
  +\hbar\BQ[\z_m,k,a][\![\hbar]\!]\bigg)\,,
\end{aligned}
\ee
where we use the equality
\be
  \prod_{\langle-b\rangle_m\neq j=0}^{m-1}(1-\z_m^{b+j})
  \=m\,.
\ee
\end{proof}

\begin{corollary}
\label{cor.binom}
For fixed $m\in\BZ_{>0}$, $b\in\BZ$, and $n\in\{0,1,\dots,m-1\}$, we have
\be
  \bigg[\!\!\begin{array}{c}mk+b\\
  m\ell+n\end{array}\!\!\bigg]_{\!q}
  \;\in\;
  \binom{k+\lfloor\tfrac{b-n}{m}\rfloor}{\ell}
  \bigg(
  \bigg[\!\!\begin{array}{c}\langle b\rangle_m\\
  n\end{array}\!\!\bigg]_{\!\z_m}
  +\hbar\BQ[\z_m,k,\ell][\![\hbar]\!]
  \bigg)\,.
\ee
\end{corollary}
\begin{proof}
This follows from the fact that
\be
\begin{small}
\begin{aligned}
  &\bigg[\!\!\begin{array}{c}mk+b\\
  m\ell+n\end{array}\!\!\bigg]_{\!q}
  \=
  \frac{(q^{m(k-\ell)+b-n+1};q)_{m\ell+n}}{(q;q)_{m\ell+n}}\\
  &\;\in\;
  \frac{(k+\lfloor\frac{b}{m}\rfloor)!}{\ell!(k-\ell+\lfloor\frac{b-n}{m}\rfloor)!}
  (-m\hbar)^{\lfloor\frac{b}{m}\rfloor-\lfloor\frac{b-n}{m}\rfloor}
  \bigg(\frac{\prod_{\langle n-b-1\rangle_m\neq j=0}^{n-1}
  (1-\z_m^{b-n+1+j})}{(\z_m;\z_m)_n}
  +\hbar\BQ[\z_m,k,\ell][\![\hbar]\!]\bigg)\\
  &\=
  \binom{k+\lfloor\tfrac{b-n}{m}\rfloor}{\ell}((k+\lfloor\tfrac{b}{m}\rfloor)
  (-m\hbar))^{\lfloor\frac{b}{m}\rfloor-\lfloor\frac{b-n}{m}\rfloor}
  \bigg(\frac{\prod_{\langle n-b-1\rangle_m\neq j=0}^{n-1}
  (1-\z_m^{b-n+1+j})}{(\z_m;\z_m)_n}
  +\hbar\BQ[\z_m,k,\ell][\![\hbar]\!]\bigg).
\end{aligned}
\end{small}
\ee
\end{proof}
  
The methods developed in~\cite{GSWZ,GW:explicit} can then be applied to prove
Theorem~\ref{thm.1.def}.

\begin{proof}[Proof of Theorem~\ref{thm.1.def}]
Lemmas~\ref{lem:pochhammers}~and~\ref{lem.poch.t1} and Corollary~\ref{cor.binom},
imply that
\be
\begin{aligned}
&f_{A,B,\nu,\ve}(x^{1/m},t^{1/m},\z_m+q-\z_m)\\
&\=
\sum_{k\in\BZ_{\geq0}^N}
\sum_{n\in(\BZ_{\geq0}\cap[0,m))^N}
q^{\frac{1}{2}Q(mk+n)-\nu(mk+n)}
\prod_{i=1}^N(-\ve_iq^{\frac{-\ve_i-1}{4}}t^{-\frac{1}{2m}})^{mk_i+n_i}
\bigg[\!\!\begin{array}{c} (mA^{t}k+A^{t}n)_i \\
k_i\end{array}\!\!\bigg]_{q}\\
&\qquad\qquad\times(q^{-(m B^{t} k+B^{t} n)_i}t;q^{-1})_{k_i}
t^{-\frac{\ve_i}{2} ((A^{t}k+A^{t}\frac{n}{m})_i + (B^{t}k+B^{t}\frac{n}{m})_i
  +\frac{1}{m})} x_i^{k_i+\frac{n_i}{m}}\\
\\
&\;\in\;
\sum_{n\in(\BZ_{\geq0}\cap[0,m))^N}
\sum_{k\in\BZ_{\geq0}^N}
\Bigg(\z_m^{\frac{1}{2}Q(n)-\nu(n)}
\prod_{i=1}^N(-\ve_i\z_m^{\frac{-\ve_i-1}{4}}t^{-\frac{1}{2m}})^{n_i}
\bigg[\!\!\begin{array}{c} \langle(A^{t}n)_i\rangle_m \\
n_i\end{array}\!\!\bigg]_{\z_m}
z_i^{\lfloor\frac{(A^{t}n)_i-n_i}{m}\rfloor}\\
&\times(\z_m^{-(B^{t} n)_i}t^{\frac{1}{m}};\z_m^{-1})_{n_i}
t^{-\frac{\ve_i}{2m} ((A^{t}n)_i + (B^{t}n)_i+1)} x_i^{\frac{n_i}{m}}
+(q-\z_m)\BQ[\z_m,t^{1/m},x^{1/m},k,\partial_t][\![q-\z_m]\!]\Bigg)\\
&\times\Bigg(\prod_{i=1}^N(-\ve_it^{-\frac{1}{2}})^{k_i}
\binom{(A^{t}k)_i+\lfloor\frac{(A^{t}n)_i-n_i}{m}\rfloor}{k_i}
(1-t)^{k_i}
t^{-\frac{\ve_i}{2}(A^{t}k+B^{t}k)_i} x_i^{k_i}\Bigg)\\
&\subseteq
R\otimes\BQ[x^{1/m},t^{1/m},\z_m][\![q-\z_m]\!]\,,
\end{aligned}
\ee
where we used the fact that $\diag(Q)\equiv(\ve+1)/2\pmod{2}$ and
$(-1)^{mk}q^{mk(mk+1)/2}=(-1)^k+\mathrm{O}(q-\z_m)$, Theorem~\ref{thm:lag.inv} and
the fact that the polynomials in $k$ induce $x$-derivatives of $z^s/\d$, which
are elements of $R$.

Since $f_{A,B,\nu,\ve}(x,t,q) \in \BZ[t^{\pm 1}, q^{\pm 1}][\![x]\!]$, it follows that
the coefficients of the powers of $q-\z_m$ in the series 
$f_{A,B,\nu,\ve}(x^{1/m},t^{1/m},\z_m+q-\z_m)$ are in
$\BZ[\z_m,t^{\pm1/m}][\![x^{1/m}]\!]$. On the other hand, we just showed they are
elements of $R\otimes\BQ[x^{1/m},t^{1/m},\z_m]$.
However,
\be
\BZ[\z_m,t^{\pm1/m}][\![x^{1/m}]\!] \cap R\otimes\BQ[x^{1/m},t^{1/m},\z_m]
\subset R[x^{1/m},t^{1/m},\z_m] \,.
\ee
It follows that 
$f_{A,B,\nu,\ve}(x^{1/m},t^{1/m},\z_m+q-\z_m) \in R[x^{1/m},t^{1/m},\z_m][\![q-\z_m[\!]$
defines a collection of series at roots of unity.
Since the coefficients in the power series expansion around $x=0$ are polynomials
in $\BZ[q^{\pm1},t^{\pm1}]$ and hence glue, we see that the collection
$f_{A,B,\nu,\ve}(x^{1/m},t^{1/m},\z_m+q-\z_m)$ also glues with a Frobenius twist.
\end{proof}

\subsection{Final proofs}
\label{sub.final}

In this last section, we give the proof of our main theorem and its corollaries.

\begin{proof}[Proof of Theorem~\ref{thm.1}]
The existence part of the map $J$ follows from Theorem~\ref{thm.1.def} and
Lemma~\ref{lem:def2A}. The topological invariance of the map $J$ follows from the
fact that $J(t,q)$ is uniquely determined by its evaluation $J(q^{n-1},q)$ for
$n \geq 1$, as follows from the injections
\be
\label{injects}
\calH_{\BZ[t^{\pm 1}, \Delta(t)^{-1}]/\BZ[t^{\pm 1}]} 
\hookrightarrow \BZ[t^{\pm 1}, \Delta(t)^{-1}][\![q-1]\!]
\stackrel{t \mapsto q^n}{\hookrightarrow} \BQ[n][\![q-1]\!] \,.
\ee
This concludes the proof of Theorem~\ref{thm.1}.
\end{proof}

Corollary~\ref{cor.cyclo} is a consequence of Lemma~\ref{lem.unique},
applied to $f(qt,q)$
instead of $f(t,q)$. 

\begin{proof}[Proof of Corollary~\ref{cor.hab}]
Since $f_{K,1}(t,q-1)$ and $J^\loop_K(t,q)$ are uniquely determined by
their evaluation at $t=q^{n-1}$ for $n>0$, and the latter agree, it follows
that $f_{K,1}(t,q-1) = J^\loop_K(t,q)$.
\end{proof}

\begin{proof}[Proof of Corollary~\ref{cor.lap1}]
First, we show that the maps~\eqref{Omdef} are well-defined. This follows from
the fact that the composite maps
\be
\calL_{\pm 1} :   \calH_{\BZ[t^{\pm 1}, \Delta(t)^{-1}]/\BZ[t^{\pm 1}]}
\hookrightarrow
\varprojlim_n
\calH_{\BZ}[t^{\pm1}]/(t^{-1};q)_n 
\to \calH_\BZ
\ee
are well defined where the map on the left is given in~\eqref{eq.unique} and
the one on the right is given in~\cite[Sec.3.1]{BCL}. 
The top map of the diagram~\eqref{cd2} is the one of Theorem~\ref{thm.1}.
The left vertical map is the Dehn-filling map.
The existence of the bottom horizontal map was shown in~\cite{BCL}, where also
the commutativity of the diagram was proven along the way. 
\end{proof}

\begin{proof}[Proof of Corollary~\ref{cor.0surgery}]
This follow from the commutativity of the outer diagram below and the injectivity of the map
$\calH_{\calO_\BK[\mathrm{disc}(\BK)^{-1}]/\BZ}\hookrightarrow \calO_\BK[\mathrm{disc}(\BK)^{-1}][\![q-1]\!]$.
\be
\begin{tikzpicture}
\draw(0,2) node {$\{\mathrm{Knots}\}$};
\draw(4,2) node {$\calH_{\BZ[t^{\pm 1}, \Delta(t)^{-1}]/\BZ[t^{\pm 1}]}$};
\draw(0,0) node {$\{M \,\, | \,\, H_1(M,\BZ)=\BZ \}$};
\draw(4,0) node {$\calH_{\calO_\BK[\mathrm{disc}(\BK)^{-1}]/\BZ}$};
\draw[->](1,2)--(2.2,2);
\draw(1.6,2.3) node {$J$};
\draw[->](0,1.5)--(0,0.5);
\draw(-0.6,1) node {$\mathrm{Dehn}_0$};
\draw[->](4,1.5)--(4,0.3);
\draw(4.3,1) node {$\calL_0$};
\draw[->](2.1,0)--(2.7,0);
\draw(4,-1) node {$\calO_\BK[\mathrm{disc}(\BK)^{-1}][\![q-1]\!]$};
\draw[->](0,-0.3)--(2.1,-1);
\draw[->](5.7,2) to[out angle = 0, in angle = 0, curve through={(6.2,1)}] (5.9,-1);
\draw(6.2,0.75) node[right] {$\calL_0$};
\draw[->](4,-0.3)--(4,-0.7);
\end{tikzpicture}
\ee
\end{proof}


\appendix

\section{The Burau representation and the Alexander polynomial}

In this appendix we discuss for completeness, some well-known facts about the
Burau representation in its original and reduced versions, and the Alexander
polynomial.

We follow the notation from Section~\ref{sub.q=1}. 
Taking the trace of $y\fb_r(\beta)$, where the first strand is colored by $\ell$
and the $i$-th strand for $i=2,\dots,r$ is colored by $\nu_i(k)+\ell$, gives 
\be
\label{app1}
\tr(S(y\fb_r(\beta)))
  \=
  \sum_{(\ell,k)\in\BZ_{\geq0}^{N+1}}
  \prod_{i=1}^N(-\ve_it^{-\frac{1}{2}})^{k_i}
  \binom{(A^{t}k)_i+\ell}{k_i}
  (1-t)^{k_i}
  t^{-\frac{\ve_i}{2} ((A^{t}k)_i + (B^{t}k)_i+2\ell)}
  y^{\nu(k)+r\ell}\,.
\ee
From MacMahon's theorem, it follows that
\be
\label{app2}
\tr(S(y\fb_r(\beta)))=\det(I-y\fb_r(\beta))^{-1} \,.
\ee
Notice that $\sum_{\ell=0}^\infty
\ell^dy^{r\ell}=\Li_{-d}(y^r)\in(1-y^r)^{-d-1}\BZ[y^r]$ is the $(-d)$-polylogarithm
that satisfies
\be
\label{app3}
  \underset{y=1}{\Res}\;
  y^c\,\Li_{-d}(ty^r)\;dy
  \=(-r)^{-1}\Big(-\frac{c+1}{r}\Big)^dt^{-\frac{c+1}{r}}\,.
\ee
This, together with~\eqref{app1} and~\eqref{app2} implies that
\be
\begin{aligned}
  &-r\;\underset{y=1}{\Res}
  \;\det(I-y\fb_r(\beta))^{-1}dy\\
  &\=
  \sum_{k\in\BZ_{\geq0}^N}
  \prod_{i=1}^N(-\ve_it^{-\frac{1}{2}})^{k_i}
  \binom{(A^{t}k)_i-\frac{\nu(k)+1}{r}}{k_i}
  (1-t)^k_i
  t^{-\frac{\ve_i}{2} ((A^{t}k)_i + (B^{t}k)_i-2\frac{\nu(k)+1}{r})} \,.
%
\end{aligned}
\ee
On the other hand, taking the trace over all but the $i$-th stand of the strand
equals to taking the trace of all but the first stand (as follows by Schur's lemma
of the corresponding state-sums). This and the definition of $f_{A,B,\nu,\ve}(t,1)$
from~\eqref{AB1} implies that for all $i=1,\dots,r$ we have:
\be
\begin{aligned}
  &\sum_{(\ell,k)\in\BZ_{\geq0}^{N+1}}^{\nu_i(k)+\ell\=\mathrm{const.}}
  \prod_{i=1}^N(-\ve_it^{-\frac{1}{2}})^{k_i}
  \binom{(A^{t}k)_i+\ell}{k_i}
  (1-t)^{k_i}
  t^{-\frac{\ve_i}{2} ((A^{t}k)_i + (B^{t}k)_i+2\ell)}\\
  &\=
  t^{i-\frac{r-1}{2}+\frac{1}{2}\sum_i\ve_i}f_{A,B,\nu,\ve}(t,1)\,.
\end{aligned}
\ee
Summing over $i=1,\dots,r$ it follows that for all integers $\mu_i \in \BZ$ for
$i=1,\dots,r$, and all $\ell \in \BZ$ we have
\be
\begin{aligned}
&\sum_{k\in\BZ_{\geq0}^N}
\prod_{i=1}^N(-\ve_it^{-\frac{1}{2}})^{k_i}
\binom{(A^{t}k)_i+\ell-\sum_{j=1}^{r-1}\mu_j\nu_j(k)}{k_i}
(1-t)^k_i
t^{-\frac{\ve_i}{2} ((A^{t}k)_i + (B^{t}k)_i+2\ell-2\sum_{j=1}^{r-1}\mu_j(\nu_j(k)))}\\
&\=
(1+\mu_1(t^{-1}-1)+\cdots+\mu_{r-1}(t^{1-r}-1))^{-1}\,t^{-\frac{r-1}{2}
  +\frac{1}{2}\sum_i\ve_i}\,f_{A,B,\nu,\ve}(t,1)\,,
\end{aligned}
\ee
Specializing to $\mu_i=1/r$ and $\ell=0$ and using~\eqref{A2} gives 
\be
\begin{aligned}
-\underset{y=1}{\Res}
\;\det(I-y\fb_r(\beta))^{-1}dy
&\=
(1+t+\cdots+t^{r-1})^{-1}
\,t^{\frac{r-1}{2}+\frac{1}{2}\sum_i\ve_i}\,f_{A,B,\nu,\ve}(t,1)\\
&\= (1+t+\cdots+t^{r-1})^{-1}
t^{\frac{r-1}{2}+\frac{1}{2}\sum_i\ve_i}
\Delta_K(t)^{-1}\\
&\=
\det(I-\fb_r^{\mathrm{red}}(\beta))^{-1}\,.
\end{aligned}
\ee


\bibliographystyle{plain}
\bibliography{biblio}

\end{document}